 \documentclass[11pt]{article}

\usepackage[text={14cm, 20cm}]{geometry}

\usepackage{pdflscape}

\usepackage{graphicx}

\usepackage{palatino}
\usepackage[mathcal]{euler}
\usepackage{extarrows}

\usepackage{amsmath,amsthm,amsfonts,amssymb,latexsym,amscd,enumerate, etoolbox}

 \usepackage{comment}
\usepackage{xy}
\xyoption{all}
\theoremstyle{plain}
    \newtheorem{theorem}{Theorem}[section]
    \newtheorem{lemma}[theorem]{Lemma}
    \newtheorem{corollary}[theorem]{Corollary}
    \newtheorem{proposition}[theorem]{Proposition}

 \theoremstyle{definition}
    \newtheorem{definition}[theorem]{Definition}
    
    \newtheorem{example}[theorem]{Example}

    \newtheorem{remark}[theorem]{Remark}

\theoremstyle{remark}

\numberwithin{equation}{section}

\DeclareMathOperator{\Ad}{Ad}

\DeclareMathOperator{\ind}{index}
\DeclareMathOperator{\ext}{ext}

\DeclareMathOperator{\End}{End}

\DeclareMathOperator{\sgn}{sgn}
\DeclareMathOperator{\ch}{ch}
\DeclareMathOperator{\Todd}{Todd}
\DeclareMathOperator{\AS}{AS}

\DeclareMathOperator{\Lie}{Lie}

\DeclareMathOperator{\Spin}{Spin}

\DeclareMathOperator{\SO}{SO}
\DeclareMathOperator{\GL}{GL}

\DeclareMathOperator{\U}{U}
 \DeclareMathOperator{\Tr}{Tr}

 \DeclareMathOperator{\Int}{int}
 
 \DeclareMathOperator{\ev}{ev}

         \DeclareMathOperator{\supp}{supp}
\DeclareMathOperator{\pt}{pt}

 \DeclareMathOperator{\even}{even}
  \DeclareMathOperator{\odd}{odd}

\begin{document}   

\newtoggle{long}
\togglefalse{long}


\newcommand{\myemph}{\emph}

\newcommand{\Spinc}{\Spin^c}

    \newcommand{\R}{\mathbb{R}}
    \newcommand{\C}{\mathbb{C}} 
    \newcommand{\N}{\mathbb{N}}
    \newcommand{\Z}{\mathbb{Z}} 
    \newcommand{\Q}{\mathbb{Q}}
    \newcommand{\bT}{\mathbb{T}}
    \newcommand{\bP}{\mathbb{P}}

\newcommand{\g}{\mathfrak{g}}
\newcommand{\h}{\mathfrak{h}}
\newcommand{\p}{\mathfrak{p}}
\newcommand{\kg}{\mathfrak{g}} 
\newcommand{\kt}{\mathfrak{t}}
\newcommand{\kA}{\mathfrak{A}}
\newcommand{\XX}{\mathfrak{X}}
\newcommand{\kh}{\mathfrak{h}} 
\newcommand{\kp}{\mathfrak{p}}
\newcommand{\kk}{\mathfrak{k}}

\newcommand{\cE}{\mathcal{E}}
\newcommand{\cS}{\mathcal{S}}
\newcommand{\cL}{\mathcal{L}}
\newcommand{\cH}{\mathcal{H}}
\newcommand{\cO}{\mathcal{O}}
\newcommand{\cB}{\mathcal{B}}
\newcommand{\cK}{\mathcal{K}}
\newcommand{\cP}{\mathcal{P}}
\newcommand{\cD}{\mathcal{D}}
\newcommand{\cF}{\mathcal{F}}
\newcommand{\cX}{\mathcal{X}}
\newcommand{\cM}{\mathcal{M}}

\newcommand{\Sj}{ \sum_{j = 1}^{\dim M}}

\newcommand{\cSM}{\cS}
\newcommand{\PM}{P}
\newcommand{\DM}{D}
\newcommand{\LM}{L}
\newcommand{\vM}{v}

\newcommand{\sigDg}{\sigma^D_g}

\newcommand{\Bigwedge}{\textstyle{\bigwedge}}

\newcommand{\ii}{\sqrt{-1}}

\newcommand{\Ubar}{\overline{U}}

\newcommand{\beq}[1]{\begin{equation} \label{#1}}
\newcommand{\eeq}{\end{equation}}
   \newcommand{\Todo}{\textbf{To do}}

\newcommand{\Wedge}{\Lambda}

\newcommand{\specialin}{\hspace{-1mm} \in \hspace{1mm} }

\title{A fixed point theorem on noncompact manifolds}

\author{Peter Hochs\footnote{University of Adelaide, \texttt{peter.hochs@adelaide.edu.au}}  \hspace{1.3mm}and Hang Wang\footnote{University of Adelaide, \texttt{hang.wang01@adelaide.edu.au}}}

\date{\today}

\maketitle

\abstract{We generalise the Atiyah--Segal--Singer fixed point theorem to noncompact manifolds. 
Using $KK$-theory, we extend the equivariant index to the noncompact setting, and obtain a fixed point formula for it.  The fixed point formula is the explicit cohomological expression from Atiyah--Segal--Singer's result. In the noncompact case, however, we show in examples that this expression yields characters of infinite-dimensional representations. 
In one example, we realise characters of discrete series representations on the regular elements of a maximal torus, in terms of the index we define. Further results are a fixed point formula for the index pairing between equivariant $K$-theory and $K$-homology, and a non-localised expression for the index we use, in terms of deformations of principal symbols. The latter result is one of several links we find to
indices of deformed symbols and operators studied by various authors.
}

\tableofcontents

\section{Introduction}

In the second part \cite{AS} of the \emph{Index of elliptic operators} series, Atiyah and Segal proved a fixed point formula for compact groups and manifolds. This allows one to compute the equivariant  index of an elliptic operator (or an elliptic complex) in terms of data on the fixed point sets of the group elements. In  \cite{AS3}, a cohomological version of this formula was obtained, which we will call the Atiyah--Segal--Singer fixed point theorem. It has applications for example in representation theory. Indeed, in \cite{AB2}, Atiyah and Bott used a fixed point formula (which equals the Atiyah--Segal--Singer theorem in the case considered) to prove the Weyl character formula.

Our goals in this paper are to generalise the Atiyah--Segal--Singer theorem to noncompact manifolds, and to apply this generalisation in relevant situations. 

\subsection*{The main result and some applications}

We define an index on possibly noncompact manifolds, that generalises the equivariant index for compact groups and manifolds (see Definition~\ref{g-index}). Assuming the fixed point set of a group element $g$ is compact, we show that this index is given by exactly the same cohomological expression as in the Atiyah--Segal--Singer theorem. This is our main result, Theorem~\ref{thm fixed pt}. We also obtain a fixed point formula for the index pairing between equivariant $K$-theory and $K$-homology in Theorem \ref{thm index pair}. In the non-equivariant setting, very general expressions for this pairing were given in \cite{CGRS}; Theorem \ref{thm index pair} is
 an equivariant version of these results for the operators considered here.

While the cohomological expression for the index is the same as in the compact case,  in noncompact examples
 we see that it gives rise to
 characters of \emph{infinite-dimensional} representations.  These can never occur as indices of elliptic operators on compact manifolds, so that the theory really gives us something new. 
For example,
we use the fixed point theorem in Subsection~\ref{sec ds} to express the character of a representation in the discrete series of a semisimple Lie group in terms of our index, on the regular elements of a maximal torus. Other examples and applications are:
\begin{itemize}
\item a holomorphic linearisation theorem, related to {\cite[Chapter 4]{GGK}} and {\cite[Theorem 7.2]{Braverman}};
\item explicit computations for actions by the circle on the plane and the two-sphere;
\item a relation with kernels of Fredholm operators, in particular Callias-type Dirac operators \cite{Anghel, BrShi, Bunke, Callias, Kucerovsky};
\item a relation with Braverman's index of deformed Dirac operators~\cite{Braverman};
\item a relative index theorem, in the spirit of {\cite[Theorem 4.18]{GL}};
\item some geometric consequences in the cases of the Hodge-Dirac and $\Spin$-Dirac operators.
\end{itemize}

In all cases we consider, we find that the index can be expressed explicitly in terms of the kernel of a deformation of the operator in question. (In the discrete series example, the operator does not even have to be deformed.) On noncompact manifolds, one can often obtain a well-defined index of a Dirac operator by applying a deformation, with suitable growth behaviour. See for example \cite{Anghel, Braverman, Braverman14, Bunke, Callias, HM, Kucerovsky, MZ, Paradan02}. This index then depends on the deformation used. While we do not use a deformation to define our index, we see in examples  that it equals an index defined via a deformation. One could speculate that this means that the index we use implicitly includes a canonical choice of such a deformation. For the Callias-type operators studied in \cite{Anghel, BrShi, Bunke, Callias, Kucerovsky}, their equivariant indices can be expressed as the index we define, plus a term representing the dependence on the deformation used,  in terms of its behaviour ``at infinity". (Previously, Callias-type operators were not studied in combination with group actions, so only non-equivariant indices were computed.)

The relation to index theory of deformed Dirac operators is strengthened in the last section of this paper, which is independent of the fixed point formula. There we give an expression for the index of elliptic operators involving deformations of their principal symbols. 

Other generalised fixed point theorems include \cite[Main Theorem 1]{BV1} and \cite[Theorem 20]{BV2} (for transversally elliptic operators), \cite[Theorem 7.5]{Braverman} (for deformed Dirac operators on noncompact manifolds), the results in \cite{DEM} (for \emph{correspondences}, generalising self-maps on manifolds), \cite[Theorem 2.7]{Emerson} (for groupoids) and \cite[Theorem 6.1]{WW} (for orbifolds).

\subsection*{Idea of the proof}

Let us sketch some technical steps involved in defining the index and proving the fixed point formula. We consider a Riemannian manifold $M$, and an elliptic operator $D$ on a vector bundle $E\to M$. Let $G$ be a compact Lie group acting on $E$, preserving $D$. Under assumptions about grading and self-adjointness, we have a class $[D]$ in the equivariant $K$-homology group $K_0^G(M)$ of $M$. Let $g\in G$. Then we may replace $G$ by the compact Abelian group generated by $g$, and still retain all information about the action by the element $g$. A localisation theorem in $K$-homology allows us to construct the \emph{$g$-index} map
\[
\ind_g\colon K_0^G(M) \to \C.
\]
The $g$-index of the operator $D$ is defined as the $g$-index of its class $[D]$ in $K_0^G(M)$.
If $M$ is compact, this is the usual equivariant index of $D$, evaluated at~$g$.

If $M$ is compact, the principal symbol $\sigma_D$ of $D$ defines a class in the equivariant topological $K$-theory group $K^0_G(TM)$. In our setting, $M$ may be noncompact. Then we have a class
\[
[\sigma_D] \in KK_G(M, TM)
\]
in the equivariant $KK$-theory of the pair $(C_0(M), C_0(TM))$. The Dolbeault--Dirac operator on $TM$ defines a class
\[
[D_{TM}] \in KK_G(TM, \pt).
\]
An index theorem by Kasparov implies that, with respect to the Kasparov product $\otimes_{TM}$ over $C_0(TM)$, we have
\[
[D] = [\sigma_D] \otimes_{TM} [D_{TM}] \quad \in KK_G(M, \pt) = K_0^G(M).
\]
This generalises the Atiyah--Singer index theorem.

The proof of the fixed point formula for the $g$-index of $D$ is a $KK$-theoretic generalisation of the proof by Atiyah and Segal for the compact case in \cite{AS}. 
This generalisation involves  Kasparov's index theorem, localisation theorems  in $KK$-theory, and $KK$-theoretic versions of the Gysin wrong-way maps in $K$-theory. Another ingredient is a class
\beq{eq sigma D g intro}
\sigma^D_g \in K^0_G(TM)_g
\eeq
associated to $\sigma_D$, 
in the equivariant topological $K$-theory of~$TM$, localised (in the algebraic sense) at $g$. 
Using these techniques, and keeping track of what happens in both entries in $KK$-theory, allows us to obtain an expression for the $g$-index of $D$ in terms of data on the fixed point set of $g$. While all constructions in the proof are $KK$-theoretic in nature, the end result is a purely cohomological expression. 
An explicit description of the class \eqref{eq sigma D g intro} in terms of 
 a deformation of the symbol $\sigma_D$ allows us to prove a non-localised expression for the $g$-index, independent of the fixed point formula.

\subsection*{Outline}

The $g$-index is introduced in Section \ref{sec fixed pt formula}. It allows us to state the fixed point formula in Theorem~\ref{thm fixed pt}. 

In Section~\ref{sec loc thm}, we prove the localisation results which imply that the $g$-index is well-defined. In Section \ref{sec Kasparov}, we review an index theorem by Kasparov. This result, and related techniques, are used
in the proof of the fixed point theorem in Section~\ref{sec pf fixed pt}.

The applications and examples mentioned above are discussed in Section~\ref{sec examples}. In Section~\ref{sec non-loc}, we obtain a
non-localised expression for the $g$-index of an elliptic operator.

\subsection*{Acknowledgements}

The authors are grateful to Gert Heckman, Bram Mesland, Yanli Song  and Mich\`ele Vergne for helpful discussions and comments. Special thanks go to Gennadi Kasparov for his explanations of the material in Section \ref{sec Kasparov}. This first author was supported by the European Union, through Marie Curie fellowship PIOF-GA-2011-299300.

\subsection*{Notation}

If $A$ is a subset of a set $B$, then we will denote the inclusion map $A\hookrightarrow B$ by $j^B_A$. We denote the one-point set by $\pt$. For any set $A$, we will write $p^A$ for the map from $A$ to $\pt$.

If $U$ is an open subset of a locally compact Hausdorff space $X$, then we denote by $k^X_U$ the inclusion map $C_0(U) \hookrightarrow C_0(X)$ defined by extending functions by zero outside $U$. If $Y$ is another locally compact Hausdorff space, we will write
\[
KK(X, Y) := KK(C_0(X), C_0(Y)),
\]
and similarly for equivariant $KK$-theory. The Kasparov product $\otimes_{C_0(X)}$ over $C_0(X)$ will also be denoted by $\otimes_X$. If $X$ has a Borel measure, and $E\to X$ is a Hermitian vector bundle, then the $*$-homomorphism $\pi_X\colon C_0(X)\to \cB(L^2(E))$ is given by the pointwise multiplication on $L^2$-sections of $E$. If $G$ is a locally compact group acting on $X$, and $H<G$ is a subgroup, we will write $G\times_H X$ for the quotient of $G\times X$ by the action by $H$ given by
\[
h\cdot (g, x) = (gh^{-1}, hx),
\]
for $h\in H$, $g\in G$ and $x\in X$.

If $M$ is a manifold, its tangent bundle projection $TM\to M$ is denoted by $\tau_M$. 
If a Riemannian metric is given, we will often tacitly use it to  identify the tangent bundle of $M$ with the cotangent bundle. 
The complexification of a vector space or vector bundle is denoted by a subscript $\C$.


\section{The fixed point formula} \label{sec fixed pt formula}

Our goal in this paper is to generalise the Atiyah--Segal--Singer fixed point theorem ({\cite[Theorem 3.9]{AS3}}, based on {\cite[Theorem 2.12]{AS}}) to noncompact manifolds, and to find interesting applications of this generalisation. This leads us to define the $g$-index in Subsection \ref{sec g index}. 
\iftoggle{long}{
This definition involves $KK$-theory, of which we give an overview in Subsection \ref{sec KK}. 
}{}
The key to defining the $g$-index is a localisation theorem, which is stated in Subsection~\ref{sec loc}. The main result of this paper is the fixed point formula in Theorem~\ref{thm fixed pt}, stated in Subsection~\ref{sec fixed pt}. This formula is entirely cohomological, and does not involve $KK$-theory. Some properties of the $g$-index are given in Subsection \ref{sec prop}. 

Throughout this paper, $M$ will be a Riemannian manifold. We consider  an isometric diffeomorphism $g$ from $M$ to itself. Suppose the closure of the powers of $g$ in the isometry group (with respect to the compact-open topology) is a compact group $G$. 
Equivalently, suppose $g$ is an element of a compact group $H$ acting isometrically on $M$, and let $G<H$ be the closed subgroup generated by $g$.
In any case, $G$ is Abelian. Let $M^g$ be the fixed point set of $g$. 

Let $E = E^+ \oplus E^- \to M$ be a $\Z_2$-graded, Hermitian vector bundle. Let $D$ be an odd, essentially self-adjoint, elliptic differential operator, with principal symbol $\sigma_D$. (For example, $D$ can be a Dirac-type operator on a complete manifold.) We will also write $D$ for the self-adjoint closure of $D$.
\iftoggle{long}
{

\subsection{$KK$-theory} \label{sec KK}

The $g$-index will be defined in terms of $KK$-theory. However, we hope that the explicit expression for this index in Theorem \ref{thm fixed pt}, and the examples in Section~\ref{sec examples}, will help to make it interesting also to readers who are less familiar with $KK$-theory. For those readers, we give a brief overview of $KK$-theory here. We omit many details, for which we refer to {\cite[Chapter~VIII]{Blackadar}}.

Let $A$ and $B$ be two $C^*$-algebras. 
The theory simplifies somewhat if $A$ is separable and $B$ is $\sigma$-unital, which we assume.
In the rest of this paper, all $C^*$-algebras we need are of the form $C_0(X)$, the algebra of continuous functions on a locally compact Hausdorff space $X$ that vanish at infinity. Such algebras are separable and $\sigma$-unital.
\begin{definition}
A \emph{Hilbert $B$-module} is defined in the same way as a complex Hilbert space, but with the complex numbers replaced by $B$. The algebra of bounded, adjointable operators on a Hilbert $B$-module $\cE$ will be denoted by $\cB(\cE)$.

An (even) \emph{Kasparov $(A, B)$-cycle} is a triple $(\cE, F, \pi)$, where
\begin{itemize}
\item $\cE$ is a $\Z_2$-graded Hilbert $B$-module;
\item $F\in \cB(\cE)$ is an odd operator;
\item $\pi\colon A\to \cB(\cE)$ is a $*$-homomorphism into the even operators,
\end{itemize}
such that for all $a\in A$, the operators $\pi(a)(F^2 - 1)$, $\pi(a)(F^*-F)$ and $[F, \pi(a)]$  are compact.

If $A = \C$, then $\pi$ has to be given by scalar multiplication. We then omit $\pi$ from the notation.

The Abelian group $KK(A, B)$ consists of classes, denoted by $[\cE, F, \pi]$, of Kasparov $(A, B)$-cycles of the form $(\cE, F, \pi)$, with respect to a homotopy relation. Addition is induced by the direct sum.

If $G$ is a locally compact group acting on $A$ and $B$ in a suitable way, then there is also an equivariant version $KK_G(A, B)$. Furthermore, there is an odd version of $KK$-theory, where there are no gradings on Hilbert modules.
\end{definition}

Some important properties of $KK$-theory are the following.
\begin{enumerate}
\item If $A=B=\C$, then $KK(\C, \C) = \Z$. If $G$ is a compact group, then $KK_G(\C, \C)=R(G)$, the representation ring of $G$.
\item If $A=\C$ and $B=C_0(X)$, for a locally compact Hausdorff space $X$, then $KK(\C, C_0(X)) = K^0(X)$, the topological $K$-theory of $X$. This is also true for equivariant $KK$-theory and $K$-theory.
\item The elliptic operator $D$ introduced at the start of this section defines an element 
\beq{eq Khom D}
[D] := \Bigl[L^2(E), \frac{D}{\sqrt{D^2 +1}}, \pi_M\Bigr]
\eeq
of the equivariant \emph{$K$-homology} group $KK_G(C_0(M), \C)$. Here $\pi_M\colon C_0(M)\to \cB(L^2(E))$ is given by pointwise multiplication.
\item If $\varphi\colon A_1 \to A_2$ and $\psi\colon B_1\to B_2$ are $*$-homomorphisms between $C^*$-algebras, then these induce
\[
\begin{split}
\varphi^*\colon &KK(A_2, B)\to KK(A_1, B); \\
\psi_*\colon &KK(A, B_1)\to KK(A, B_2). \\
\end{split}
\]
This construction is functorial, and extends to the equivariant case.
\item If $A_1, A_2, A_3$ and $B_1, B_2$ are $C^*$-algebras, then there are products
\[
\begin{split}
KK(A_1, B_1) \times KK(A_2, B_2) &\xrightarrow{\otimes} KK(A_1\otimes A_2, B_1 \otimes B_2);\\
KK(A_1, A_2) \times KK(A_2, A_3) &\xrightarrow{\otimes_{A_2}} KK(A_1,A_3),\\
\end{split}
\]
called the exterior and interior Kasparov products, respectively.
These products are associative, and compatible with functorially induced maps in all natural ways.
\end{enumerate}

If $A=C_0(X)$ and $B = C_0(Y)$ for locally compact Hausdorff spaces $X$ and $Y$, then we will write
\[
KK(X, Y) := KK(C_0(X), C_0(Y)),
\]
and denote the product $\otimes_{C_0(X)}$ by $\otimes_X$. We use analogous notations in the equivariant case. Note that, since $C_0$ is a contravariant functor, continuous, proper maps $\eta\colon X_1\to X_2$ and $\theta\colon Y_1\to Y_2$ between locally compact Hausdorff spaces induce
\[
\begin{split}
\eta_*\colon &KK(X_1, Y)\to KK(X_2, Y); \\
\theta^*\colon &KK(X, Y_2)\to KK(X, Y_1). \\
\end{split}
\]
}
{
Then we have the  element 
\beq{eq Khom D}
[D] := \Bigl[L^2(E), \frac{D}{\sqrt{D^2 +1}}, \pi_M\Bigr]
\eeq
of the equivariant \emph{$K$-homology} group $KK_G(M, \pt) := KK_G(C_0(M), \C)$. Here $\pi_M\colon C_0(M)\to \cB(L^2(E))$ is given by pointwise multiplication. For background material on $KK$-theory, see \cite[Chapter~VIII]{Blackadar}.
}

\subsection{Localisation}\label{sec loc}

%

Let $R(G)_g := R(G)_{I_g}$ be the localisation of the representation ring $R(G)$ at 
the prime ideal
\[
I_g := \{\chi \in R(G); \chi(g) = 0\}.
\]
For any module $\cM$ over $R(G)$, we write $\cM_g := \cM_{I_g}$ for the corresponding localised module over $R(G)_g$. Similarly, if $m\in \cM$, and $\varphi\colon \cM \to \cM'$ is a module homomorphism to another such module, we write $m_g \in \cM_g$ and 
\[
\varphi_g\colon \cM_g \to \cM'_g
\]
for the corresponding localised versions.

For any two $G$-$C^*$ algebras $A$ and $B$, the group $KK_G(A, B)$ is a module over the  ring $R(G) = KK_G(\C, \C)$, via the exterior Kasparov product. Fix a $G$-$C^*$ algebra $A$.
The inclusion map 
\[
j_{M^g}^M\colon M^g \hookrightarrow M
\]
 induces
\[
(j_{M^g}^M)^*_g\colon KK_G(A, C_0(M))_g \to KK_G(A, C_0(M^g))_g.
\]
\begin{theorem} \label{thm loc 1} If $A$ is separable, 
the map $(j_{M^g}^M)^*_g$  is an isomorphism of Abelian groups. This is still true if $M\setminus M^g$ is a manifold, rather than all of $M$.
\end{theorem}
\begin{remark}
If $A = \C$, then this reduces to {\cite[Theorem 1.1]{AS}}. We need this more general statement, because in the noncompact case, principal symbols define classes in $KK_G(C_0(M), C_0(TM))$ as in \eqref{eq sigma D}, rather than in $KK_G(\C, C_0(TM))$ when $M$ is compact.
\end{remark}

We will also use an analogue of Theorem \ref{thm loc 1} for the first entry in $KK$-theory. Its formulation is slightly more subtle.
\begin{theorem} \label{thm loc 2}
Suppose that $M^g$ is compact and that $A$ is $\sigma$-unital. Let $U, V \subset M$ be two $G$-invariant, relatively compact open neighbourhoods of $M^g$, such that $\Ubar \subset V$. Then the map
\[
\bigl((j^V_{\Ubar})_*\bigr)_g\colon KK_G(C_0(\Ubar), A)_g \to KK_G( C_0(V), A)_g
\]
is an isomorphism of Abelian groups. This is still true if $M\setminus M^g$ is a manifold, rather than all of $M$.
\end{theorem}

Theorems \ref{thm loc 1} and \ref{thm loc 2} will be proved in Section \ref{sec loc thm}, for graded $KK$-theory, i.e.\ the combination of even and odd $KK$-theory. We will only apply the even versions, however. The cases where only $M \setminus M^g$ is a manifold were included because we will also apply Theorem \ref{thm loc 2} to one-point compactifications of manifolds.


\subsection{The $g$-index} \label{sec g index}

Suppose the fixed point set $M^g$ is compact. Let $U, V$ be as in Theorem~\ref{thm loc 2}. Consider the proper map $p^{\Ubar}\colon \Ubar \to \pt$, and the inclusion map $k_V^M\colon C_0(V) \to C_0(M)$ given by extending function by zero outside $V$. Let $A$ be a $\sigma$-unital $G$-$C^*$ algebra.
By  Theorem~\ref{thm loc 2}, we have the maps
\begin{multline} \label{eq ind g}
KK_G(C_0(M), A)_g \xrightarrow{(k^M_V)^*_g} KK_G(C_0(V), A)_g \xrightarrow{((j^V_{\Ubar})_*)_g^{-1}} KK_G(C(\Ubar), A)_g \\\xrightarrow{(p^{\Ubar}_*)_g} KK_G(\C, A)_g.
\end{multline}
\begin{lemma} \label{lem U V}
The composition \eqref{eq ind g} is independent of the sets $U$ and $V$.
\end{lemma}
\begin{proof}
To prove independence of $U$, let $U'$ be a $G$-invariant, relatively compact neighbourhood of $M^g$ such that $\overline{U'}\subset U$. Then we have the commutative diagram
\[
\xymatrix{
 & KK_G(C(\overline{U'}), A) \ar[dr]^-{p^{\overline{U'}}_*} \ar[dl]_-{(j^V_{\overline{U'}})_*}  \ar[d]^-{(j^{\Ubar}_{\overline{U'}})_*}& \\
 KK_G(C_0(V), A)  & KK_G(C(\Ubar), A) \ar[r]_-{p^{\Ubar}_*} \ar[l]^-{(j^V_{\Ubar})_*} & KK_G(\C, A).
}
\]
Commutativity of this diagram implies that
\[
 (p^{\overline{U'}}_*)_g\circ \bigl((j^V_{\overline{U'}})_*\bigr)_g^{-1} = (p^{\Ubar}_*)_g\circ \bigl((j^V_{\Ubar})_*\bigr)_g^{-1}. 
\]
So \eqref{eq ind g} is indeed independent of $U$.

To prove independence of $V$, let $V'$ be a $G$-invariant, relatively compact open subset of $M$ containing $V$. Then the following diagram commutes:
\[
\xymatrix{
& KK_G(C_0(V'), A) \ar[d]^-{(k^{V'}_V)^*}& \\
KK_G(C_0(M), A) \ar[r]_-{(k^M_V)^*} \ar[ur]^-{(k^M_{V'})^*}  & KK_G(C_0(V), A)  & KK_G(C(\Ubar), A). \ar[l]^-{(j^V_{\Ubar})_*} \ar[ul]_-{(j^{V'}_{\Ubar})_*} 
}
\]
Therefore, we have
\[
\bigl((j^{V'}_{\Ubar})_*\bigr)_g^{-1} \circ (k^M_{V'})^*_g = \bigl((j^V_{\Ubar})_*\bigr)_g^{-1} \circ (k^M_V)^*_g,
\]
so that \eqref{eq ind g} is independent of $V$.
\end{proof}

To define the $g$-index, we only need the case of Lemma \ref{lem U V} where $A = \C$. Later we will also use the general case, however.

Let 
\[
\ev_g\colon R(G) \to \C
\]
be defined by evaluating characters at $g$, i.e.\  $\ev_g(\chi) := \chi(g)$, for $\chi \in R(G)$. 
In view of Lemma~\ref{lem U V}, we obtain a well-defined index as follows.
\begin{definition}
\label{g-index}
The \emph{$g$-index} is the map
\[
\ind_g\colon KK_G(M, \pt) \to \C
\]
defined as the composition
\beq{def:g-index}
KK_G(M, \pt) \hookrightarrow KK_G(M, \pt)_g \xrightarrow{(p^{\Ubar}_*)_g \circ ((j^V_{\Ubar})_*)_g^{-1} \circ (k^M_V)^*_g} 
KK_G(\pt, \pt)_g \cong
R(G)_g \xrightarrow{(\ev_g)_g}\C.
\eeq
We will write
\[
\ind_g(D) := \ind_g[D], 
\]
where
$[D] \in KK_G(M, \pt)$
is the class \eqref{eq Khom D}.
\end{definition}
Note that $(k^M_V)^*_g[D]_g$ is simply the localisation at $g$ of the $K$-homology class of the restriction of $D$ to $V$.

\begin{remark}
The $g$-index of $D$ could also have been called the $D$-Lefschetz number of $g$.
\end{remark}

\subsection{Properties of the $g$-index} \label{sec prop}

If $M$ is \emph{compact}, then we may take $U = V = M$ in Definition \ref{g-index}. Furthermore, the map $p^M\colon M\to \pt$ is proper. In that case, the composition \eqref{eq ind g} simply equals the map
\[
(p^M_*)_g\colon KK_G(C_0(M), A)_g \to KK_G(\C, A)_g.
\]
If $A = \C$, then it follows that for compact $M$, the $g$-index of $D$ equals 
\begin{equation} \label{eq ind g cpt}
\ind_g(D) = \ind_G(D)(g),
\end{equation}
 the usual equivariant index of $D$ evaluated at $g$. Note that on the right hand side of \eqref{eq ind g cpt}, $G$ can be any compact Lie group acting isometrically on $M$, if the action lifts to $E$, commutes with $D$, and contains $g$.

In general, however, the $g$-indices on noncompact manifolds give us something more general than the equivariant index in the compact case. In the examples in Section \ref{sec examples}, we will see that the $g$-index can be used to describe characters of infinite-dimensional representations. These cannot be realised as indices on compact manifolds. And even on compact manifolds, an equivariant index can be decomposed into $g$-indices which individually correspond to infinite-dimensional representations. See Subsection~\ref{sec T S2}.


The $g$-index has an excision property.
\begin{lemma} \label{lem excision}
Let $V$ be a $G$-invariant, relatively compact, open neighbourhood of $M^g$. Suppose there is a $G$-equivariant open embedding $V\hookrightarrow \tilde M$ into a  $G$-manifold $\tilde M$. Suppose the action by $G$ on $\tilde M$ has no fixed points outside $V$.
Suppose there is a Hermitian, $\Z_2$-graded $G$-vector bundle $\tilde E \to \tilde M$ and an odd, self-adjoint, elliptic differential operator $\tilde D$ on $\tilde E$ such that $\tilde E|_V = E|_V$ and $\tilde D|_V = D|_V$. 
Then 
\[
\ind_g(D) = \ind_g(\tilde D).
\]
\end{lemma}
\begin{proof}
By Proposition 10.8.8 in \cite{HR}, we have
\[
 (k^{M}_V)^*[D]=(k^{\tilde M}_V)^*[\tilde D]\in KK_G(V, \pt).
 \]
This implies the claim.
\end{proof}

\begin{example}
Suppose $M$ has a $G$-equivariant $\Spin$-structure, and let $D$ be the $\Spin$-Dirac operator. Let $M\hookrightarrow \tilde M$ be a $G$-equivariant open embeddng into a compact $G$-manifold $\tilde M$ with a $G$-equivariant $\Spin$-structure. If $G$ is connected and $\ind_g(D) \not=0$, then $g$ must have a fixed point in $\tilde M\setminus M$. Indeed, Atiyah and Hirzebruch \cite{AH} showed that the $g$-index of the $\Spin$-Dirac operator on $\tilde M$ is zero in this case. So the claim follows by Lemma \ref{lem excision}.
\end{example}

Another property of the $g$-index is multiplicativity. Let $D_1$ and $D_2$ be operators like $D$ on manifolds $M_1$ and $M_2$ respectively, and consider the product operator 
\[
D_1 \times D_2 := D_1 \otimes 1 + 1\otimes D_2
\]
on $M_1 \times M_2$ (where graded tensor products are used). Then functoriality of the Kasparov product implies that
\[
\ind_g(D_1 \times D_2) = \ind_g(D_1)\ind_g(D_2).
\]
In the index theory of deformed Dirac operators developed in \cite{Braverman}, the deformation used means that an analogous multiplicativity property is highly nontrivial, see \cite{HS, MZ, Paradan02}.

\subsection{Fixed points} \label{sec fixed pt}

Having defined the $g$-index, we can state the main result of this paper. We will use the fact that the connected components of the fixed point set $M^g$ are smooth submanifolds of $M$, possibly of different dimensions.

Since $M^g$ is compact, the restriction to $TM^g$ of the principal symbol $\sigma_D$ of $D$ defines a class
\beq{eq sigma D g}
[\sigma_D|_{TM^g}] \in KK_G(\pt, TM^g)
\eeq
 Let $N\to M^g$ be the union of the normal bundles to each of the components of $M^g$. 
Consider the topological $K$-theory class 
\beq{eq wedge N}
\bigl[\Bigwedge N_{\C}\bigr] := \left[ \bigoplus_{j} {\Bigwedge}^{2j} N \otimes \C \right] - \left[ \bigoplus_{j} {\Bigwedge}^{2j+1} N \otimes \C \right] \in KK_G(\pt, M^g).
\eeq

For any trivial $G$-space $X$, we have
\[
KK_G(\pt, X) \cong KK(\pt, X)\otimes R(G).
\]
We can evaluate  the factor in $R(G)$ of any class $a\in KK_G(\pt, X)$ at $g$, to obtain
 $a(g) \in KK(\pt, X)\otimes \C$. In this way,
evaluating 
 the classes \eqref{eq sigma D g} and \eqref{eq wedge N} at $g$ gives the classes
\[
[\sigma_D|_{TM^g}] (g) \in KK(\pt, TM^g) \otimes \C
\]
and
\beq{eq wedge N g}
\bigl[\Bigwedge N_{\C}\bigr](g)  \in KK(\pt, M^g) \otimes \C,
\eeq
respectively. 

Consider the Chern characters 
\[
\begin{split}
\ch \colon& KK(\pt, TM^g)\to H^*(TM^g);\\
\ch \colon& KK(\pt, M^g)\to H^*(M^g),
\end{split}
\]
defined on each smooth component of $M^g$ separately. By {\cite[Lemma~2.7]{AS}}, the $K$-theory class \eqref{eq wedge N g} is invertible, hence so is its Chern character. An explicit expression for the inverse 
\[
 \frac{1}{\ch\bigl(\bigl[\Bigwedge N_{\C}\bigr](g)\bigr)} \quad \in H^*(M^g)\otimes \C
\]
of this element is given in~{\cite[eq.~(3.8)]{AS3}}.
The cohomology group $H^*(M^g)$ acts on $H^*(TM^g)$ via the pullback along the tangent bundle projection~$\tau_{M^g}$. Let $\Todd(TM^g \otimes \C)$ be the cohomology class on $M^g$ obtained by putting together the Todd-classes of the complexified tangent bundles to all components of $M^g$.
\begin{theorem}[Fixed point formula]\label{thm fixed pt}
The $g$-index of $D$ equals
\begin{equation}
\label{eq:main.thm.coho.formula}
\ind_g(D)=\int_{TM^g} \frac{\ch \bigl( [\sigma_D|_{TM^g}](g)\bigr) \Todd(TM^g\otimes\C)}{\ch\bigl(\bigl[\Bigwedge N_{\C}\bigr](g)\bigr)}.
\end{equation}
\end{theorem}
The integral in this expression is the sum of the integrals over all connected components of $TM^g$ of the integrand corresponding to each component.

If $M$ is compact, then 
 \eqref{eq ind g cpt} implies that Theorem \ref{thm fixed pt} reduces to the Atiyah--Segal--Singer fixed point formula \cite[Theorem 3.9]{AS3}.

\subsection{The index pairing}\label{sec index pair}

In the course of the proof of Theorem \ref{thm fixed pt}, we will also find a fixed point formula for the index pairing
 (i.e.\ the Kasparov product)
\[
KK_G(\pt, M) \times KK_G(M, \pt) \to KK_G(\pt, \pt).
\]
Note that any element of the equivariant topological $K$-theory group $KK_G(\pt, M)$ can be represented by a formal difference
$[F_0] - [F_1]$, for two $G$-equivariant vector bundles $F_0, F_1\to M$ that are equal outside a compact set. We will write $F := F_0 \oplus F_1$, with the 
 $\Z_2$-grading where $F_0$ is the even part and $F_1$ the odd part, and $[F] := [F_0] - [F_1] \in KK_G(\pt, M)$.
\begin{theorem}[Fixed point formula for the index pairing] \label{thm index pair}
We have
\[
([F] \otimes_M [D])(g) = \int_{TM^g} \frac{\ch \bigl([F|_{M^g}](g) \bigr)\ch \bigl( [\sigma_D|_{TM^g}](g)\bigr) \Todd(TM^g\otimes\C)}{\ch\bigl(\bigl[\Bigwedge N_{\C}\bigr](g)\bigr)}.
\]
\end{theorem}
Recall that $M^g$ was assumed to be compact, and that we use the action by the cohomology of $M^g$ on the cohomology of $TM^g$ via the pullback along $\tau_{M^g}$.

Theorem 3.33 in \cite{CGRS} is a non-equivariant index formula for the index pairing in a more general context. Theorem \ref{thm index pair} is an equivariant version of this result, for operators like $D$.

The proof of Theorem \ref{thm index pair} is simpler than that of Theorem \ref{thm fixed pt}, because it does not involve localisation in the first entry of $KK$-theory. Theorem \ref{thm fixed pt} is needed for the examples and applications in Section \ref{sec examples}, such as the relation with characters of dicrete series representations. The reason for this is that Theorem \ref{thm fixed pt} provides an expression for an index of the operator $D$ itself, without the need to twist it by a $K$-theory class.


\section{Localisation} \label{sec loc thm}

We now turn to a proof of Theorems \ref{thm loc 1} and \ref{thm loc 2}. This involves certain module structures discussed in Subsection \ref{sec mod}, which are used to prove vanishing results in Subsection \ref{sec van}

In this section, we will consider \emph{graded} $KK$-theory, i.e.\ the direct sum of even and odd $KK$-theory.

\subsection{Module structures}\label{sec mod}

In this subsection only, let $G$ be any locally compact group.

\begin{proposition} \label{prop module}
Let $H<G$  be a compact subgroup. Let $Y$ be a locally compact, Hausdorff, proper $G$-space for which there is an equivariant, continuous  map $Y \to G/H$. Then for any $G$-$C^*$-algebra $A$, the groups
\[
KK_G(A, C_0(Y)) \quad \text{and} \quad KK_G(C_0(Y), A) 
\]
have  structures of a unital $R(H)$-modules.
\end{proposition}

Proposition \ref{prop module} follows from the fact that vector bundles, even on noncompact spaces,  define classes in $KK$-theory in the following way. This is probably well-known, but we include a proof for completeness' sake.

Let $X$ be a locally compact Hausdorff space, on which a locally compact group $G$ acts properly. Let $E\to X$ be a Hermitian $G$-vector bundle. The space $\Gamma_0(E)$ of continuous sections of $E$ vanishing at infinity is a right Hilbert $C_0(X)$-module by pointwise multiplication and inner products. Let $\pi_X\colon C_0(X) \to \cB(\Gamma_0(E))$ be given by pointwise multiplication.
\begin{lemma}\label{lem KK vb}
The triple
\begin{equation} \label{eq KK vb}
(\Gamma_0(E), 0, \pi_X)
\end{equation}
is a $G$-equivariant Kasparov $(C_0(X), C_0(X))$-cycle.
\end{lemma}
We will denote the class in $KK_G(X, X)$ defined by \eqref{eq KK vb} by $[E]$.
\begin{proof}
We will show that for all $f\in C_0(X)$, the operator $\pi_X(f)$ on $\Gamma_0(E)$ is compact. This implies the claim.

Let $U\subset X$ be a relatively compact open subset admitting an orthonormal frame $\{e_1, \ldots, e_r\}$ of $E|_U$. 
Let $s\in \Gamma_0(E)$.
Then
\[
s|_U = \sum_{j=1}^r (e_j, s)_E e_j.
\]
Here $(\relbar, \relbar)_E$ is the metric on $E$.
So if $f\in C_0(X)$ is supported inside $U$, then 
\[
\pi_X(f)s = \sum_{j=1}^r (e_j, fs) e_j = \sum_{j=1}^r (\bar fe_j, s) e_j.
\]
By extending the sections $e_j$ outside $U$ to elements of $\Gamma_0(E)$, we find that
 $\pi_X(f)$ is a finite-rank operator. 

For a general $f\in C_c(X)$, there is a finite open cover $\{U_j\}_{j=1}^n$ of $\supp(f)$ such that every set $U_j$ admits a local orthonormal frame for $E$. Let $\{\varphi_j\}_{j=1}^n$ be functions such that $\supp(\varphi_j)\subset U_j$, and $\sum_{j=1}^n \varphi_j$ equals one on $\supp(f)$. Then, by the preceding argument,
\[
\pi_X(f) = \sum_{j=1}^n\pi_X(\varphi_j f)
\] 
is a finite-rank operator. Hence for all $f \in C_0(X)$,  the operator $\pi_X(f)$ on $\Gamma_0(E)$ is indeed compact.
\end{proof}


Now consider the situation of Proposition \ref{prop module}.
Let $p\colon Y \to G/H$ be an equivariant, continuous map.
Let $V$ be a finite-dimensional representation space of $H$. We have the $G$-vector bundles
\[
G\times_H V \to G/H
\]
and 
\[
E_V := p^*(G\times_H V) \to Y.
\]
By Lemma \ref{lem KK vb}, this vector bundle defines a class
\[
[E_V] \in KK_G(Y, Y).
\]
\begin{lemma} \label{lem ring homom}
The map from $R(H)$ to $KK_G(Y, Y)$ given by 
\[
[V] \mapsto [E_V]
\]
as above, is a ring homomorphism.
\end{lemma}
\begin{proof}
This follows from the fact that, in the setting of Lemma \ref{lem KK vb}, for any two Hermitian $G$-vector bundles $E, E' \to X$, one has
\[
[E]\otimes_{X}[E'] = [E\otimes E'].
\]
\end{proof}

The ring homomorphism of Lemma \ref{lem ring homom} defines the module structures sought in Proposition \ref{prop module}, which has therefore been proved. If $A = \C$ and $Y$ is compact, the $R(H)$-module structure on $KK_G(\C, C_0(Y))$ defined in this way  is the one used in \cite{AS}.

\subsection{Vanishing results} \label{sec van}

We will prove  Theorems \ref{thm loc 1} and \ref{thm loc 2} by generalising  Atiyah and Segal's proof of~{\cite[Theorem~1.1]{AS}}. 
An important  step is the following generalisation of~{\cite[Corollary 1.4]{AS}}.
\begin{proposition} \label{prop loc}
Let $H<G$ be a closed subgroup such that $g \not\in H$. Let $Y$ be a compact $G$-space for which there is an equivariant map $Y \to G/H$ and $A$ a $G$-$C^*$-algebra. Then
\[
KK_G(A, C_0(Y))_g = KK_G(C_0(Y), A)_g = 0.
\]
\end{proposition}
\begin{proof}
By {\cite[Corollary~1.3]{AS}}, we have $R(H)_g = 0$. As Atiyah and Segal argued below that corollary, it is therefore enough to show that $KK_G(A, C_0(Y))$ and $KK_G( C_0(Y), A)$ are unital $R(H)$-modules. Hence the claim follows from Proposition \ref{prop module}.
\end{proof}

We will deduce Theorems \ref{thm loc 1} and \ref{thm loc 2} from the following special cases.
\begin{proposition} \label{prop loc zero}
In the setting of Theorem \ref{thm loc 1}, suppose $g$ has no fixed points in $M$. Then, if $A$ is separable, we have
\begin{equation} \label{eq loc zero 1}
KK_G(A, C_0(M))_g = 0. 
\end{equation}
If $A$ is $\sigma$-unital, then for all $G$-invariant, relatively compact open subsets $U\subset M$,
\begin{equation} \label{eq loc zero 2}
KK_G(C_0(U), A)_g = 0.
\end{equation}
\end{proposition}
If $A = \C$, then  \eqref{eq loc zero 1} is precisely {\cite[Proposition~1.5]{AS}}. By a  generalisation of the arguments in {\cite[Section~1]{AS}}, we will
 deduce Proposition~\ref{prop loc zero} from Proposition~\ref{prop loc}.

By Palais' slice theorem {\cite[Proposition 2.2.2]{Palais}}, there is an open cover $\{U_j\}_{j=1}^{\infty}$ of $M$ by $G$-invariant open sets such that for all $j$,
\[
\overline{U_j} \cong G\times_{H_j} \overline{S_j}
\]
(via the action map),
for the stabiliser $H_j < G$ of a point in $U_j$, and an $H_j$-invariant submanifold $S_j \subset M$. Suppose that $g$ has no fixed points. Then it does not lie in any of the stabilisers $H_j$. Therefore, Proposition \ref{prop loc} implies that
\[
KK_G(A, C_0(\overline{U_j}))_g = KK_G(C_0(\overline{U_j}), A)_g = 0.
\]

Let $X\subset M$ be any $G$-invariant, compact subset. The proof of Proposition~\ref{prop loc zero} is based on the following fact.
\begin{lemma} \label{lem loc cpt}
If $A$ is separable, then
\begin{equation} \label{eq loc cpt 1}
KK_G(A, C_0(X))_g= 0.
\end{equation}
If $A$ is $\sigma$-unital, then
\beq{eq loc cpt 2}
KK_G(C_0(X), A)_g = 0.
\eeq
\end{lemma}
\begin{proof}
We will use an induction argument based on exact sequences in $KK$-theory. We work out the details for \eqref{eq loc cpt 1}. Then \eqref{eq loc cpt 2} can be proved in the same way, with exact sequences in the second entry in $KK$-theory replaced by the corresponding exact sequences in the first entry. The conditions that $A$ is separable or $\sigma$-unital imply that these exact sequences exist.

For $j, n\in \N$, write $X_j := \overline{U_j} \cap X$, and $Y_n := X_1 \cup\cdots \cup X_n$. Fix $n\in \N$, and consider the exact sequence of $C^*$-algebras
\[
0 \to C_0(X_{n+1}\setminus Y_n) \to C_0(X_{n+1}) \to C_0(X_{n+1}\cap Y_n) \to 0.
\]
It induces the exact triangle 
\[
\xymatrix{
KK_G(A, C_0(X_{n+1})) \ar[r] & KK_G(A, C_0(X_{n+1}\cap Y_n)) \ar[d]^-{\partial} \\
	& KK_G(A, C_0(  X_{n+1}\setminus Y_n )). \ar[ul]
}
\]
(See e.g.\ {\cite[Theorem 19.5.7]{Blackadar}}.) By Proposition \ref{prop loc}, we have
\[
KK_G(A, C_0(X_{n+1}))_g = KK_G(A, C_0(X_{n+1}\cap Y_n))_g = 0.
\]
Since localisation at $g$ preserves exactness, we find that
\begin{equation} \label{eq KK Xn}
 KK_G(A, C_0(  X_{n+1}\setminus Y_n ))_g = 0.
\end{equation}

Using the exact sequence
\[
0 \to C_0(Y_{n+1}\setminus Y_n) \to C_0(Y_{n+1}) \to C_0(Y_n) \to 0
\]
in a similar way, we obtain the exact triangle
\[
\xymatrix{
KK_G(A, C_0(Y_{n+1}))_g \ar[r] & KK_G(A, C_0( Y_n))_g \ar[d]^-{\partial} \\
	& KK_G(A, C_0(  Y_{n+1}\setminus Y_n ))_g. \ar[ul]
}
\]
Since $Y_{n+1}\setminus Y_n = X_{n+1}\setminus Y_n$, the vanishing of \eqref{eq KK Xn} implies that
\[
KK_G(A, C_0(Y_{n+1}))_g = KK_G(A, C_0(Y_n))_g. 
\]
Because $Y_1 = X_1$, Proposition \ref{prop loc} implies that
\[
KK_G(A, C_0(Y_1))_g = 0. 
\]
Since $X$ is compact, it can be covered by finitely many of the sets $X_j$.
Hence the claim follows by induction on $n$.
\end{proof}

\noindent\emph{Proof of Proposition \ref{prop loc zero}.}
Let $U\subset M$ be a $G$-invariant, relatively compact open subset. Consider the exact sequence
\[
0 \to C_0(U) \to C_0(\Ubar) \to C_0(\partial U) \to 0.
\]
If $A$ is $\sigma$-unital, this induces the localised exact triangle
\[
\xymatrix{
KK_G(C_0(\Ubar), A)_g  \ar[dr]& KK_G( C_0(\partial U), A)_g   \ar[l] \\
	& KK_G(C_0(U), A)_g. \ar[u]_-{\partial}
}
\]
Lemma \ref{lem loc cpt} implies that
\[
KK_G(C_0(\Ubar), A)_g =
KK_G(C_0(\partial{U}), A)_g  = 0.
\]
So we find that
$
 KK_G(C_0(U), A)_g = 0.
$

Similarly, If $A$ is separable, we have the  exact triangle
\[
\xymatrix{
KK_G(A, C_0(\Ubar))_g \ar[r] & KK_G(A, C_0(\partial U))_g \ar[d]^-{\partial} \\
	& KK_G(A, C_0(U))_g. \ar[ul]
}
\]
Applying Lemma \ref{lem loc cpt} in the same way, we find that
$
KK_G(A, C_0(U))_g  = 0.
$
The equality \eqref{eq loc zero 1}  follows, because $M$ is the direct limit of sets $U$ as above, and because $KK$-theory commutes with direct limits in the second entry. 
\hfill $\square$

\begin{remark}
The reason why \eqref{eq loc zero 2} does not hold if $U$ is replaced by $M$, and hence why Theorem \ref{thm loc 2} has to be stated more subtly than Theorem \ref{thm loc 1}, is that $KK$-theory does not commute with direct limits in the first entry. For example, the domain of the analytic assembly map in the Baum--Connes conjecture \cite{BCH} is the \emph{representable $K$-homology} group
\[
RK_*^G(X) := \varinjlim_{Y\subset X; \text{ $Y/G$ cpt}} KK_G(C_0(Y), \C),
\]
for a locally compact Hausdorff space $X$ on which a locally compact group $G$ acts properly. This does not equal the usual $K$-homology group $KK_G(C_0(X), \C)$ in general.
\end{remark}

\subsection{Proofs of localisation results} \label{sec pf loc}

\noindent \emph{Proof of Theorem \ref{thm loc 1}.}
Consider the exact sequence
\[
0 \to C_0(M \setminus M^g) \to C_0(M) \xrightarrow{(j^M_{M^g})^*} C_0(M^g) \to 0.
\]
It induces the exact triangle 
\[
\xymatrix{
KK_G(A, C_0(M)) \ar[r]^{(j^M_{M^g})^*} & KK_G(A, C_0(M^g)) \ar[d]^-{\partial} \\
	& KK_G(A, C_0(M\setminus M^g)). \ar[ul]
}
\]
After localisation at $g$, the first part of
Proposition \ref{prop loc zero} yields the exact triangle
\begin{samepage}
\[
\xymatrix{
KK_G(A, C_0(M))_g \ar[r]^-{(j^M_{M^g})^*_g} & KK_G(A, C_0(M^g))_g \ar[d]^-{\partial} \\
	& 0. \ar[ul]
}
\]
\hfill $\square$
\end{samepage}

\medskip
\noindent \emph{Proof of Theorem \ref{thm loc 2}.}
Let $U$ and $V$ be as in Theorem~\ref{thm loc 2}. Similarly to the proof of Theorem \ref{thm loc 1}, we have an exact triangle
\[
\xymatrix{
KK_G(C_0(V), A)_g  \ar[dr] & KK_G( C_0(\Ubar), A)_g   \ar[l]_-{((j^V_{\Ubar})_*)_g}  \\
	& KK_G(C_0(V\setminus \Ubar), A)_g. \ar[u]_-{\partial}	
}
\]
Because 
 $V\setminus \Ubar$ is a relatively compact subset of $M\setminus M^g$, the second part of Proposition \ref{prop loc zero} implies that the bottom localised $KK$-group in this triangle equals zero.
\hfill $\square$



\section{Kasparov's index theorem} \label{sec Kasparov}

In the proof of the Atiyah--Segal--Singer fixed point theorem, the Atiyah--Singer index theorem is used to relate topological and analytical indices to each other. In the noncompact case discussed in this paper, a roughly similar role is played by an index theorem of Kasparov. We state Kasparov's index theorem in Subsection~\ref{sec index Kas}.   In Subsection~\ref{sec Bott elt}, we discuss the fibrewise Bott element for the normal bundle of a submanifold in $KK$-theory, which is dual to the class of the Dolbeault--Dirac operator in a sense. This Bott element will play an important role in the proof of Theorem \ref{thm fixed pt}. In Subsection~\ref{sec:AS.index}, we show how the Bott element can be used to deduce the Atiyah--Singer index theorem from Kasparov's index theorem in the compact case. (The main step in the argument used there will be used in the proof of Theorem \ref{thm fixed pt}.)

Most of the material in this section is based on \cite{AS3, Kas13}, and explanations to the authors by Kasparov. Although the results here are not ours, we found it worthwhile to include the details, because they have not appeared in print yet.


\subsection{The index theorem} \label{sec index Kas}

To state the theorem, we recall the definition of the Dolbeault operator class
\beq{eq DTM}
[D_{TM}] \in KK_G(TM, \C)
\eeq
in~\cite[Definition 2.8]{Kas13}. The tangent bundle $T(TM)$ of $TM$ has a natural almost complex structure $J$. For $m\in M$ and $v\in T_mM$, we have
\[
T_v(TM) = T_mM \oplus T_v(T_mM) = T_mM \oplus T_mM.
\]
With respect to this decomposition, the almost complex structure $J$ is given by the matrix $\left[ \begin{array}{cc} 0 & 1\\-1 & 0\end{array} \right]$. 
Let  $\bar \partial + \bar\partial^*$ be the Dolbeault--Dirac operator on smooth sections of the vector bundle $\Bigwedge^{0,*}T^*(TM)\to TM$, for this almost complex structure. We will identify this vector bundle with $\tau_M^*\Bigwedge TM_{\C}\rightarrow TM$.
The class \eqref{eq DTM} is the class of this operator, as in \eqref{eq Khom D}. In our arguments however, it will be more convenient to use the $\Spinc$-Dirac operator $D_{TM}$, on the same vector bundle.  This defines the same $K$-homology class as $\bar \partial + \bar\partial^*$.

\iftoggle{long}{
\begin{example}
If $M=\R^2$, then the $\Spinc$-Dirac operator on $T\R^2$ is
\begin{equation}
\label{eq:Dolbeault.TR^2}
D_{T\R^2}=e_1\frac{\partial}{\partial \xi_1}+e_2\frac{\partial}{\partial \xi_2}+\ii\epsilon_1\frac{\partial}{\partial x_1}+\ii\epsilon_2\frac{\partial}{\partial x_2}
\end{equation}
Here $(x_1, x_2; \xi_1, \xi_2)$ are the standard coordinates on $T\R^2 = \R^2\times \R^2$, and for $j=1,2$,
\[
e_j :=\ext(d\xi_j)-\Int(d\xi_j)\qquad \epsilon_j :=\ext(d\xi_j)+\Int(d\xi_j).
\]
(Here $\ext$ and $\Int$ denote exterior multiplication and interior contraction, respectively.)
We have
\[
D_{T\R^2}
=\begin{bmatrix}0 & D_{T\R^2}^{-} \\ D_{T\R^2}^{+} & 0\end{bmatrix},
\]
where $D_{T\R^2}^{-}=(D_{T\R^2}^{+})^*$ and, with $z_j :=\xi_j+\ii x_j$,
\begin{equation*}
D_{T\R^2}^{+}
=\begin{bmatrix}-\frac{\partial}{\partial z_1} & -\frac{\partial}{\partial\overline{z_2}} \\ -\frac{\partial}{\partial z_2} & \frac{\partial}{\partial\overline{z_1}} \end{bmatrix}.
\end{equation*}
This agrees with the fact that the $\Spinc$-Dirac operator equals the Dolbeault--Dirac operator on K\"ahler manifolds.
\end{example}
}{}

\begin{definition}
The \emph{topological index} is the map
\[
\ind_t\colon KK_G(M, TM) \to KK_G(M, \pt)
\]
given by the Kasparov product with $[D_{TM}]$.
\end{definition}


Consider the principal symbol
$\tilde \sigma_D := \frac{\sigma_D}{\sqrt{\sigma_D^2+1}}$
of the operator $\frac{D}{\sqrt{D^2 + 1}}$. 
For $f\in C_0(M)$, we have for all $m\in M$ and $v\in T_mM$,
%
%
%
\[
f(m)\bigl(1-\tilde \sigma_D(v)^2\bigr)= f(m) \bigl(\sigma_D(v)^2+1 \bigr)^{-1}.
\]
Since the operator $D$ is elliptic and of positive order, this expression tends to zero as $m$ or $v$ tends to infinity. It therefore defines a compact operator on the Hilbert $C_0(TM)$-module $\Gamma_0(\tau_{M}^*E)$, analogously to the proof of Lemma \ref{lem KK vb}. Therefore, the triple
\begin{equation}
\label{eq:symbol.KK-cycle}
\bigl(\Gamma_0(\tau_{M}^*E), \tilde \sigma_D, \pi_{TM} \circ \tau_M^* \bigr)
\end{equation}
is a $G$-equivariant Kasparov $(C_0(M), C_0(TM))$-cycle. 
Here $\pi_{TM}\colon C_b(TM) \to \cB(\Gamma_0(\tau_M^*E))$ is given by pointwise multiplication.
Denote by
\beq{eq sigma D}
[\sigma_D] \in KK_G(M, TM).
\eeq
the class of~(\ref{eq:symbol.KK-cycle}).
In view of the following lemma, this symbol class is a natural generalisation of the $K$-theory symbol class defined in~\cite{AS1} when $M$ is compact.

\begin{lemma} \label{lem norm sigma D}
If $M$ is compact, consider the map $p^M$ from $M$ to a point. The image  
\[
p^M_*[\sigma_D] \in K^*_G(TM)
\]
is the usual symbol class.
\end{lemma}
\begin{proof}
Since $\pi_{TM} \circ \tau_M^* \circ (p^M)^*$ is the representation of $\C$ in $\Gamma_0(\tau^*_{M}E)$ by scalar multiplication, we have
\[
p^M_*[\sigma_D] = \bigl[\Gamma_0(\tau^*_{M}E), \tilde \sigma_D \bigr] \quad \in KK_G(\pt, TM).
\]
This is corresponds to the class 
\[
\bigl[\sigma_{D^{+}}\colon \tau^*_{M}E^+ \to \tau^*_{M} E^- \bigr] \in K^0_G(TM)
\]
in the sense of~{\cite[Ch.~III eq.~(1.7)]{LM}}, where $TM$ is identified with the open unit ball bundle $BM$ over $M$. (Restricting $\sigma_{D^{+}}$ to $BM$ and then identifying $BM \cong TM$
 has the same effect as normalising $\sigma_{D^{+}}$.) 
The lemma is then proved.
\end{proof}

We conclude this subsection by stating Kasparov's index theorem, which will be used to obtain a cohomological formula for the $g$-index.

\begin{theorem}[Kasparov's index theorem {\cite[Theorem 4.2]{Kas13}}] \label{thm index}
The $K$-homology class of the elliptic operator in~(\ref{eq Khom D}) is equal to the topological index of its symbol class~(\ref{eq sigma D}), i.e.,
\begin{equation} \label{eq index Kas}
[D] = \ind_t[\sigma_D] \in KK_G(M, \C).
\end{equation}
\end{theorem}  

\begin{remark}
In \cite[Theorem 4.2]{Kas13}, the operator in question is assumed to be properly supported, which is not true for the operator $\frac{D}{\sqrt{1+D^2}}$ in general. However, let $\{\chi_j\}_{j=1}^{\infty}$ be sequence of  $G$-invariant, compactly supported functions, such that $\{\chi_j^2\}_{j=1}^{\infty}$ is a partition of unity. (This exists since $G$ is compact.) Then the operator
\[
\sum_{j=1}^{\infty} \chi_j \frac{D}{\sqrt{1+D^2}} \chi_j
\]
is properly supported, and also satisfies the other assumptions of \cite[Theorem 4.2]{Kas13}. Since this operator defines the same $K$-homology class as $\frac{D}{\sqrt{1+D^2}}$, we can apply  \cite[Theorem 4.2]{Kas13} to the class of the latter operator in this way.
\end{remark}

\iftoggle{long}
{
\begin{example}
\label{ex:Kas.ind}
Consider $M=\R^2$ and $G=\bT^1$, acting on $M$ by rotations. (The action by $G$ is not very important in this example. It is included here because it will be used in Subsection \ref{sec:S^1.on.R^2}.) Let $E=\R^2\times\C^2$ be the spinor bundle, where the fibre at every point $x = (x_1, x_2)\in\R^2$ is graded by $E_{x}=E_{x}^+\oplus E_{x}^{-}\cong\C\oplus\C.$ 
Consider the $\bT^1$-equivariant Dolbeault--Dirac operator 
\begin{equation}
\label{eq:Dolbeault.R^2}
D=\begin{bmatrix}0 & -\frac{\partial}{\partial x_1}+\ii\frac{\partial}{\partial x_2} \\ \frac{\partial}{\partial x_1}+\ii\frac{\partial}{\partial x_2} & 0 \end{bmatrix}\colon \Gamma_0(\R^2, \C^2)\rightarrow\Gamma_0(\R^2, \C^2),
\end{equation}
with symbol
\begin{equation}
\label{eq:symbol.Dol.R^2}
\sigma_D(x, \xi)=\begin{bmatrix}0 & -\ii\xi_1-\xi_2 \\ \ii\xi_1-\xi_2 & 0\end{bmatrix},
\end{equation}
for $(x, \xi)  \in T\R^{2}$.
Let $\beta$ be the Bott generator of $K^0(\R^2)$, 
which we now view as an element of 
$KK_{\bT^1}(\pt, \R^2)$. 
Then $[\sigma_D]$ is the image of $\beta$ under the map
\[
\tau_{C_0(\R^2)}\colon KK_{\bT^1}(\pt, \R^2)\rightarrow KK_{\bT^1}(\R^2, \R^4)
\] 
define by taking tensor products with $C_0(\R^2)$, as in
 {\cite[17.8.5]{Blackadar}.}
The Dolbeault class $[D_{T\R^2}]\in KK_{\bT^1}(\R^4, \pt)$ in~(\ref{eq:Dolbeault.TR^2}) is the exterior Kasparov product of $[D]\in KK_{\bT^1}(\R^2, \pt)$ 
given by~(\ref{eq:Dolbeault.R^2}) and the Dirac operator $[D_{\xi}]\in KK_{\bT^1}(\R^2, \pt)$ on tangent spaces, given by 
\[
D_{\xi}=\begin{bmatrix}0 & -\frac{\partial}{\partial \xi_1}+\ii\frac{\partial}{\partial \xi_2} \\ \frac{\partial}{\partial \xi_1}+\ii\frac{\partial}{\partial \xi_2} & 0 \end{bmatrix}.
\]
By the definition of the exterior Kasparov product, we have 
\[
[D_{T\R^2}]=(\tau_{C_0(\R^2)}[D_{\xi}])\otimes_{\R^2}[D].
\] 
By the associativity of the Kasparov product and {\cite[Proposition~18.9.1(c)]{Blackadar}}, the right hand side of~(\ref{eq index Kas}) is therefore equal to 
\[
[\sigma_D]\otimes_{\R^4}[D_{T\R^2}]=\bigl(\tau_{C_0(\R^2)}(\beta\otimes_{\R^2}[D_{\xi}]) \bigr)\otimes_{\R^2}[D].
\]
The product $\beta\otimes_{\R^2}[D_{\xi}]$ is represented by a harmonic oscillator which is a Fredholm operator on $\R^2$ with index $1$. 
Therefore, $[\sigma_D]\otimes_{\R^4}[D_{T\R^2}]=[D]$, verifying Theorem~\ref{thm index} in this example.
\end{example}
}{}

%

\subsection{The Bott element}\label{sec Bott elt}

If $S$ is a closed (as a topological subspace, i.e.\ not necessarily compact), $G$-invariant submanifold of $M$, then the Dolbeault operator classes on $TS$ and on a tubular neighbourhood of $TS$ in $TM$ are related by a \emph{(fibrewise) Bott element}. This is a technical tool that will be used several times in the paper. The material here is analogous to Definition 2.6 and Theorem 2.7 in \cite{Kas13}.

Consider the tangent bundle projections
\[
\begin{split}
\tau_S\colon & TS \to S;\\
\tau_N\colon& TN \to N.
\end{split}
\]
Denote by $\pi\colon N\to S$ the normal bundle of $S$ in $M$. 
Let $T\pi\colon TN\to TS$ be the tangent map of $\pi$. It again defines a vector bundle.
The following diagram commutes:
\beq{eq TN to S}
\xymatrix{
TN \ar[r]^-{\tau_N}\ar[d]^-{T\pi} & N \ar[d]^-{\pi}\\
TS \ar[r]^-{\tau_S} & S.
}
\eeq
This defines a vector bundle $TN \to S$. 
Consider the vector bundle
\[
\Bigwedge \widetilde{N}_{\C} := T\pi^* \bigl(\tau_S^* \Bigwedge N \otimes \C\bigr) \to TN.
\]

Let $s\in S$. Then
\[
\begin{split}
(TN)_s &:= T\pi^{-1}(\tau_S^{-1}(s)) \\
 &= \tau_N^{-1}(N_s) \\
 &= T_sS \times N_s \times N_s.
\end{split}
\]
Let $w\in (TN)_s$, and let $(\eta, \zeta) \in N_s \times N_s$ be the projection of $w$ according to this decomposition. Note that 
\[
\bigl(\Bigwedge \widetilde{N}_{\C}\bigr)_{w} = \Bigwedge N_s \otimes \C. 
\]
We define the vector bundle endomorphism $B$ of $\Bigwedge \widetilde{N}_{\C}$ by
\[
B_{w} = \ext(\zeta+\ii\eta)-\Int(\zeta+\ii\eta),
\]
for all $s$, $w$, $\eta$ and $\zeta$ as above.
Here  $\ext$ denotes the wedge product, and $\Int$ denotes contraction. With respect to the grading 
of $
 \Bigwedge \widetilde{N}_{\C}
$
by even and odd exterior powers, the operator $B$ is odd.

As $B$ is fibrewise selfadjoint, we have the bounded operator 
$B(1+B^2)^{-\frac12}$ on $\Gamma_0(TN, \Bigwedge \widetilde{N}_{\C})$.
The space $\Gamma_0(TN, \Bigwedge \widetilde{N}_{\C})$ is a right Hilbert $C_0(TN)$-module in the usual way, with respect to pointwise multiplication by functions and the pointwise inner product. Consider the representation
\[
\tilde \pi_{TS} := \pi_{TN} \circ T\pi^*\colon C_0(TS)\to \cB\bigl(\Gamma_0(TN, \Bigwedge\widetilde{N}_{\C})\bigr),
\]
where $\pi_{TN}$ is given by pointwise multiplication by functions in $C_b(TN)$.
\begin{lemma}\label{lem:fiberwise.Bott.ele}
The triple
\beq{eq def beta}
\left( \Gamma_0(TN, \Bigwedge\widetilde{N}_{\C}), B(1+B^2)^{-\frac12}, \tilde \pi_{TS}\right)
\eeq
is a $G$-equivariant Kasparov $(C_0(TS), C_0(TN))$-cycle.
\end{lemma}
\begin{proof}
Let  $f\in C_0(TS)$. Since 
$B(1+B^2)^{-\frac12}$ is a vector bundle endomorphism, it commutes with $\tilde \pi_{TS}(f)$.
Moreover, we have for all $w \in (TN)_s$ as above,
\[
\left( \tilde \pi_{TS}(f)\left(1-\left[B(1+B^2)^{-\frac12}\right]^2\right) \right)_{w}=\frac{f(v)}{1+\|\eta\|^2+\|\zeta\|^2},
\]
with $v := T\pi(w) \in T_sS$. 
This defines a function in $C_0(TN)$, and hence acts on the Hilbert $C_0(TN)$-module $\Gamma_0(TN, \Bigwedge \widetilde{N}_{\C})$ as a compact operator.
As $g$ preserves the metric  $TN$, the operator $B(1+B^2)^{-\frac12}$ is $G$-equivariant.
\end{proof}

\begin{definition}\label{def:fiberwise.Bott.ele}
The \emph{(fibrewise) Bott element} of the normal bundle $N\to S$ is the class
\[
\beta_N \in KK_G(TS, TN)
\] 
of the cycle \eqref{eq def beta}.
\end{definition}

\subsection{The Bott element and Dolbeault classes}

The Bott element is useful to us because of the following property. This was used in \cite[bottom paragraph on p.\ 30]{Kas13}; we work out some details of the proof in this subsection.
\begin{proposition}
\label{lem:pairing.B.D}
Under the Kasparov product $KK_G(TS, TN)\times KK_G(TN, \pt)\rightarrow KK_G(TS, \pt)$, one has
\[
\beta_N\otimes_{TN} [D_{TN}]=[D_{TS}].
\]
\end{proposition}

To prove this proposition, one can use the part $D_1$ of the $\Spinc$-Dirac operator $D_{TN}$ acting in the fibre directions of $TN \to TS$. 
 For $s\in S$ and $v\in T_sS$, we have $T\pi^{-1}(v) = N_s \oplus T_vN$. Let $a$ be the rank of $N$, and let  $\{f_1, \ldots, f_a\}$ be a local orthonormal frame of $N\to S$. This defines coordinate functions
$\kappa_j$ and $\lambda_j$ on the parts $N_s$ and $T_vN$ of the fibres $T\pi^{-1}(v)$ of $TN\to TS$, respectively. 
For  $j = 1, \ldots, a$, consider the vector bundle endomorphisms
\[
e_j :=\ext(f_j)-\Int(f_j)\quad \text{and} \quad \epsilon_j :=\ext(f_j)+\Int(f_j)
\]
of $\Bigwedge N\otimes \C\to S$, pulled back along \eqref{eq TN to S} to endomorphisms of $\Bigwedge \widetilde{N}_{\C}\to TN$.
Then $D_1$ is the operator  
\[
D_1 := \sum_{j=1}^a e_j\frac{\partial}{\partial\kappa_j}+\ii\epsilon_j\frac{\partial}{\partial\lambda_j}
\]
on $\Gamma^{\infty}(TN, \Bigwedge \widetilde{N}_{\C})$. This can be viewed as a  family of operators  on the fibres of  $TN \to TS$.

It defines a class in $KK$-theory as follows. 
Let $\Gamma_c(TN, \Bigwedge\widetilde N_{\C})$ be the space of continuous compactly supported sections of $\Bigwedge\widetilde N_{\C}$. Let $\cE_0$ be the completion of this space  into a Hilbert $C_0(TS)$-module 
 with respect to the following $C_0(TS)$-valued inner product:
\beq{eq:C_0TS.inner.product}
\langle f, h\rangle (v):=\int_{T\pi^{-1}(v)}\overline{f(t)}h(t) \, dt, 
\eeq
for $f, h \in \Gamma_c(TN, \Bigwedge\widetilde N_{\C})$ and $v\in TS$. 
The  operator $D_1$   gives rise to the class
\beq{eq:D_1.cycle}
[D_1]:= \bigl[\cE_0, D_1(1+D_1^2)^{-\frac12}, \pi_{TN}\bigr] \in KK_G(TN, TS).
\eeq

\begin{lemma} \label{lem D1 DTS}
We have
\[
[D_1]\otimes_{TS} [D_{TS}] = [D_{TN}] \quad \in KK_G(TN, \pt).
\]
\end{lemma}
\begin{proof}
Regarding $N$ as an open subset of $M$, we identify their tangent bundles when restricted to $S$, i.e., $TN|_S=TM|_S.$
Therefore, as vector bundles over $TN$, we have
\begin{align*}
\Bigwedge\widetilde N_{\C}\otimes T\pi^*\tau_S^*\Bigwedge TS_{\C}
&=T\pi^*\tau_S^*\Bigwedge N_{\C}\otimes T\pi^*\tau_S^*\Bigwedge TS_{\C}\\
&=T\pi^*\tau_S^*\Bigwedge (N\oplus TS)_{\C}=T\pi^*\tau_S^*\Bigwedge (TM|_S)_{\C}\\
&=T\pi^*\tau_S^*\Bigwedge (TN|_S)_{\C}
=\tau_N^*\Bigwedge TN_{\C}.
\end{align*}
The last equality follows from commutativity of~(\ref{eq TN to S}).
Thus, as Hilbert spaces with representations of $C_0(TN)$, 
\beq{eq:identify.tensor}
\cE_0\otimes_{C_0(TS)}L^2(TS, \tau_S^{*}\Bigwedge TS_{\C})\cong L^{2}(TN, \tau_N^*\Bigwedge TN_{\C}).
\eeq 

Under this identification, we have
\[
D_1 \otimes 1 + 1\otimes D_{TS} = D_{TN}.
\]
(Here we use graded tensor products.)
Consider the bounded operator 
\begin{equation}
\label{eq:product}
F:=\frac{D_1\otimes1+1\otimes D_{TS}}{\sqrt{1+D_1^2\otimes 1+1\otimes D_{TS}^2}}
\end{equation}
on $\cE_0\otimes_{C_0(TS)}L^2(TS, \tau_S^{*}\Bigwedge TS_{\C})$. 
Write
\[
F_1 := \frac{D_1}{\sqrt{1+D_1^2}}; \qquad
F_{TS} := \frac{D_{TS}}{\sqrt{1+D_{TS}^2}}.
\]
We can verify that $F$ is an $1\otimes F_{TS}$-connection, and the graded commutator $[F_1\otimes 1, F]$ is positive modulo compact operators.  
Therefore, by~{\cite[Definition~18.4.1]{Blackadar}} the Kasparov product $[D_1]\otimes_{TS} [D_{TS}] $ is represented by the operator $F$ on  $L^{2}(TN, \tau_N^*\Bigwedge TN_{\C})$.
The lemma is then proved.
\end{proof}

\begin{lemma} \label{lem betaN D1}
The product
\[
\beta_N \otimes _{TN} [D_1] \quad \in KK_G(TS, TS).
\]
is the ring identity.
\end{lemma}
\begin{proof}
The idea is that in this product, we pair fibrewise Bott classes and Dolbeault classes, and thus obtain the trivial line bundle over $TS$. 
To see this, observe first the isomorphism 
\beq{eq:beta.D_1.module}
\Gamma_c(TN, \Bigwedge\widetilde N_{\C})\otimes_{C_c(TN)}\Gamma_c(TN, \Bigwedge\widetilde N_{\C})\cong\Gamma_c(TN, \Bigwedge\widetilde N_{\C}\otimes\Bigwedge\widetilde N_{\C})
\eeq
as $C_c(TS)$-modules.
Denote by $\cE'$ the completion of the right-hand side under the $C_0(TS)$-valued inner product defined in a similar way as~(\ref{eq:C_0TS.inner.product}).
It can be checked that 
\beq{eq def F0}
F_0:=\frac{B\otimes 1+1\otimes D_1}{\sqrt{1+B^2\otimes 1+1\otimes D_1^2}}
\eeq
 is a $1\otimes \frac{D_1}{\sqrt{1+D_1^2}}$-connection, and that for all $a$ in $C_0(TS)$, the operator 
 \[
 \tilde\pi_{TS}(a)\bigl[B(1+B^2)^{-\frac12}\otimes 1, F_0\bigr]\tilde\pi_{TS}(a)^*
 \]
 is positive modulo compact operators on $\cE'$. 
Hence, the Kasparov product of $\beta_N$, given by~(\ref{eq def beta}), and the class $[D_1]$, given by~(\ref{eq:D_1.cycle}), is equal to
\beq{eq cycle betaN D1}
\left[\cE', F_0, \tilde\pi_{TS}\right]\in KK_G(TS, TS).
\eeq
As in the proof of~{\cite[Theorem~2.7 (2)]{Kas13}}, we
apply the rotation homotopy 
\[
F_t := \frac{(B + \sin(t)D_1)\otimes 1+1\otimes \cos(t)D_1}{\sqrt{1+(B^2+ \sin(t)^2D_1^2)\otimes 1+1\otimes \cos(t)^2D_1^2}},
\]
for $t\in [0, \pi/2]$.
Then the operator $F_0$ in the cycle \eqref{eq cycle betaN D1} is transformed into $F_{\pi/2} = F'\otimes 1$, where 
\[
F':=(B+D_1)(1+B^2+D_1^2)^{-\frac12}.
\]
In summary, $\beta_N$ and $[D_1]$ are families of operators indexed by $TS$ whose Kasparov product is represented by $F'$.
At every $v\in T_sS$, the square of $B + D_1$ is the harmonic oscillator operator
\[
\sum_{j=1}^a\bigl(\kappa_j^2 +\lambda_j^2-\frac{\partial}{\partial \kappa_j^2} -\frac{\partial^2}{\partial \lambda_j^2}\bigr)+2\deg-a
\]
on $T\pi^{-1}(v) \cong N_s\times N_s$. (Here $\deg$ is the degree in $\Bigwedge N$.) It 
has a one dimensional kernel, concentrated in degree zero, generated by 
\beq{eq ker DB}
(\eta, \zeta) \mapsto e^{-\frac{\|\eta\|^2+\|\zeta\|^2}{2}} \quad \in C_0(N_s\times N_s).
\eeq
Thus, over each fibre, $F'$ is a Fredholm operator with index $1$, and $\beta_N\otimes_{TN}[D_1]$ is equal to the exterior product of this Fredholm operator in $KK_G(\C, \C)$ and $[C_0(TS), 0, \pi_{TS}]\in KK_G(TS, TS)$ both representing the respective ring identities. Hence the claim follows.
\end{proof}

\medskip\noindent\emph{Proof of Proposition \ref{lem:pairing.B.D}.}
Using Lemmas \ref{lem D1 DTS} and \ref{lem betaN D1}, and associativity of the Kasparov product, we find that
\[
\beta_N\otimes_{TN}[D_{TN}]=\bigl(\beta_N\otimes_{TN}[D_1] \bigr)\otimes_{TS}[D_{TS}]= [D_{TS}].
\]
This finishes the proof.
\hfill $\square$
\medskip

We will later need the restriction of the Bott element to $TS$. Consider the class
\[
\bigl[\tau_{S}^*\Bigwedge N_{\C} \bigr] := \left[ \bigoplus_{j} {\Bigwedge}^{2j} \tau_{S}^*N \otimes \C \right] - \left[ \bigoplus_{j} {\Bigwedge}^{2j+1} \tau_{S}^*N \otimes \C \right] \in KK_G(TS, TS),
\]
defined as in Lemma \ref{lem KK vb}.
\begin{lemma} \label{lem res Bott}
We have
\[
(j^{TN}_{TS})^*\beta_N = \bigl[\tau_{S}^*\Bigwedge N_{\C}\bigr] \quad \in KK_G(TS, TS).
\]
\end{lemma}
\begin{proof}
The Hilbert $C_0(TS)$-module in $(j^{TN}_{TS})^*\beta_N$ is $\Gamma_0(TS, \tau^*_S\Bigwedge N_{\C})$.
Because $B|_{TS}$ is the zero operator, the claim follows.
\end{proof}


\subsection{The Atiyah--Singer index theorem}
\label{sec:AS.index}

Suppose for now that  $M$ is compact and $G$ is trivial.
Then Kasparov's index theorem reduces to the Atiyah--Singer index theorem, see {\cite[Remark 4.5]{Kas13}}. We provide the details of this implication here, because these will used in the proof of Theorem~\ref{thm fixed pt}. 

Consider the Atiyah--Singer topological index map
\[
\ind_t^{\AS}\colon KK(\pt, TM) \to \Z,
\]
which maps a class $\sigma \in KK(\pt, TM)$ to
\begin{equation}
\label{eq:AS.top.index}
\int_{TM}\ch(\sigma)\Todd(TM\otimes\C).
\end{equation}
Note that we do not have the factor $(-1)^{\dim M}$ in~(\ref{eq:AS.top.index}) as in {\cite[Theorem 2.12]{AS3}}, because we use a different almost complex structure on $TM$ than in {\cite[p.~554]{AS3}}, giving the opposite orientation. 

\begin{lemma} \label{lem ind top cpt}
As a map 
\[
KK(\pt, TM) \to KK(\pt, \pt),
\]
right multiplication by $[D_{TM}]$ is the Atiyah--Singer topological index.
\end{lemma}

Because of Lemma \ref{lem ind top cpt}, Theorem \ref{thm index} implies the Atiyah--Singer index theorem. Indeed,
since $M$ is compact, the map $p^M\colon M \to \pt$ is proper. By functoriality of the Kasparov product, Lemma \ref{lem ind top cpt} implies that
the following diagram commutes:
\[
\xymatrix{
KK(M, TM) \ar[rr]^-{\ind_t} \ar[d]^-{p^M_*}& & KK(M, \pt)  \ar[d]^-{p^M_*}\\
KK(\pt, TM) \ar[rr]^-{\ind_t^{\AS}}& & \Z = KK(\pt, \pt).
}
\]
By Lemma \ref{lem norm sigma D},
applying the map $p^M_*$ to both sides of \eqref{eq index Kas}, and using commutativity of the above diagram, one obtains the Atiyah--Singer index theorem.

\medskip
\noindent \emph{Proof of Lemma \ref{lem ind top cpt}.}
The proof is a reformulation of the arguments in~\cite{AS3}, using $KK$-theory.



There are embeddings $M\hookrightarrow \R^n$ with normal bundle $N$ of rank $a$, and $TM\hookrightarrow T\R^n = \C^n$ with normal bundle $TN$. As $N$ is homeomorphic to a tubular neighbourhood of $M$ in $\R^n$,  we can identify $TN$ with an open neighbourhood of $TM$ in $\C^n.$ (Note that here, the submanifold $S$ of $M$ in  Subsection \ref{sec Bott elt} is replaced by the submanifold $M$ of $\R^n$.) 

Denote by 
\[
\beta_{N} \in KK(TM, TN) 
\]
the fibrewise Bott element over $TM$ in $TN$, in the sense of Definition~\ref{def:fiberwise.Bott.ele}.
Then by Proposition~\ref{lem:pairing.B.D},
\beq{eq decomp DM}
[D_{TM}]=\beta_{N}\otimes[D_{TN}].
\eeq
The Chern character is compatible with the pairing of $K$-theory and $K$-homology. 
The Chern character of 
 the Bott generator $\beta$ of $K^0(\R^2)$ 
 \iftoggle{long}
 {
 in Example~\ref{ex:Kas.ind} 
 }{}
 is a generator of $H^2(\R^2)$. As the Dolbeault class $[D_{\R^2}]$ of $\R^2$ is dual to $\beta$, its Chern character is the Poincar\'e dual of $\ch(\beta)$. So $\ch[D_{\R^2}]$ is the fundamental class $[\R^2]$ of $\R^2$. 
Similarly, working with  the exterior Kasparov product of $n$ copies of $\beta$, we conclude that $\ch[D_{\R^{2n}}]=[\R^{2n}]$. 
Noting that $T\R^n=\R^{2n}$, then by functoriality of the Chern character, we have 
\beq{eq fund TN}
\ch[D_{TN}]=\ch \bigl( (k_{TN}^{T\R^n})_*[D_{T\R^n}] \bigr) =(k_{TN}^{T\R^n})_*\ch[D_{T\R^{n}}]=(k_{TN}^{T\R^n})_*[T\R^{n}]=[TN].
\eeq
Thus, the Chern character of $[D_{TN}]$ is the fundamental class $[TN] \in H_{2n}(TN)$.
Let $\sigma \in KK(\pt, TM)$ be given. Then \eqref{eq decomp DM} and \eqref{eq fund TN} imply that
\begin{equation}
\label{eq:sigma.D_TM}
\sigma \otimes_{TM} [D_{TM}]=\int_{TN}\ch(\sigma)\wedge\ch (\beta_{N}).
\end{equation}
The Thom isomorphism $\psi_{TN}\colon H^*(TM)\rightarrow H^*(TN)$ (mapping between compactly supported cohomologies) is an isomorphism of $H^*(TM)$-modules. So we can 
rewrite the integral~(\ref{eq:sigma.D_TM}) as
\begin{multline}
\label{eq:cal.int}
\int_{TN}\ch(\sigma) \wedge \ch(\beta_{N}) =\int_{TM}\psi_{TN}^{-1}(\ch(\sigma)\wedge\ch(\beta_{N}))\\
=\int_{TM}\ch(\sigma)\wedge\psi_{TN}^{-1}(\ch(\beta_{N})).
\end{multline}
To calculate $u:= \psi_{TN}^{-1}(\ch(\beta_{TN}))$, we make use of the following diagram: 
\[
\xymatrix{
K^{*}(TM)  \ar[r]^{\psi_{TN}} \ar[d]_{\ch}&  K^{*}(TN) \ar[r]^{(j^{TN}_{TM})^*} \ar[d]_{\ch} & K^{*}(TM) \ar[d]_{\ch}  \\
H^*(TM) \ar[r]^{\psi_{TN}} & H^*(TN) \ar[r]^{(j^{TN}_{TM})^*} & H^*(TM).}
\]
Note that in the second line, the composition is equal to the exterior product by the Euler class $e(TN).$
In the above diagram, we have by Lemma \ref{lem res Bott},
\[
\xymatrix{
&  \beta_{N} \ar@{|->}[r]^-{(j^{TN}_{TM})^*} \ar@{|->}[d]_{\ch} & \sum_j(-1)^{j}\Bigwedge^j TN \ar@{|->}[d]_{\ch}  \\
u\ar@{|->}[r]^-{\psi_{TN}} & \ch (\beta_{N}) \ar@{|->}[r]^-{(j^{TN}_{TM})^*} & u\cdot e(TN).}
\]
As the above square commutes by functoriality of the Chern character,  and since $TN=\tau_M^*N_{\C}$ and $N_{\C}\oplus(TM\otimes\C)=M\times \C^n$, we obtain 
\[
u=\frac{\ch\bigl( \sum_j(-1)^{j}\Bigwedge^j TN \bigr)}{e(TN)}=\tau_{M}^*\left(\frac{e(TM)}{\ch\bigl( \sum_j(-1)^{j}\Bigwedge^j TM)\bigr)}\right)=\tau_M^*\bigl(\Todd(TM\otimes\C)\bigr).
\]
Therefore, together with~(\ref{eq:sigma.D_TM}) and~(\ref{eq:cal.int}) one has
\[
\sigma_{TM}\otimes[D_{TM}]=\int_{TM}\ch(\sigma)\Todd(TM\otimes\C),
\]
and the lemma is proved.
\hfill $\square$


\section{Proof of the fixed point formula} \label{sec pf fixed pt}

After proving Theorems \ref{thm loc 1} and \ref{thm loc 2} and discussing Kasparov's index theorem, we are ready to prove Theorem \ref{thm fixed pt}.

We start in Subsection \ref{sec Gysin}, by generalising Gysin maps, or wrong-way functoriality maps in $K$-theory, that play a key role in \cite{AS1}. We use these generalised Gysin maps in Subsection \ref{sec Gysin loc} to set up the commutative diagrams we need. We discuss a map defined by evaluating characters at $g$ in Subsection \ref{sec eval}. Then we introduce a class in the topological $K$-theory of $TM$, localised at $g$, defined by the principal symbol of $D$. The properties of that class allow us to finish the proof of Theorem \ref{thm fixed pt}.

\subsection{Gysin maps}\label{sec Gysin}


Let $S \subset M$ be a $G$-invariant submanifold, with inclusion map $j^M_S\colon S \hookrightarrow M$. (In the applications of what follows, $S$ will be a connected component of the fixed point set $M^g$.)
Let $N \to S$ be the normal bundle of $S$ in $M$. 
The inclusion map $j^{TN}_{TS}\colon TS\hookrightarrow TN$  induces  a map 
\[
(j^{TN}_{TS})^*\colon C_0(TN)\rightarrow C_0(TS)
\] 
by restriction.
We identify $TN$ with an open neighbourhood of $TS$ in $TM$, via a $G$-equivariant embedding
$
TN \hookrightarrow TM.
$
 In this way, we have the  injective map
\[
k^{TM}_{TN}\colon C_0(TN) \hookrightarrow C_0(TM),
\] 
defined by extending functions by zero. 

\begin{definition}
Let $A$ be any $G$-$C^*$ algebra. The map $(j^{TM}_{TS})_{!}\colon KK_G(A, C_0(TS)) \to KK_G(A, C_0(TM))$ is the composition
\[
KK_G(A, C_0(TS)) \xrightarrow{\relbar \otimes_{C_0(TS)} \beta_N} KK_G(A, C_0(TN)) \xrightarrow{(k^{TM}_{TN})_*} KK_G(A, C_0(TM)).
\]
Here $\beta_N \in KK_G(TS, TN)$ is the Bott element, as in Definition~\ref{def:fiberwise.Bott.ele}.
\end{definition}

We also have the usual map
\[
(j^{TM}_{TS})^*\colon KK_G(A, C_0(TM)) \to KK_G(A, C_0(TS)).
\]
\begin{lemma} \label{lem pull push}
The map
\[
(j^{TM}_{TS})^* \circ(j^{TM}_{TS})_{!}\colon KK_G(A, C_0(TS)) \to KK_G(A, C_0(TS))
\]
is given by the Kasparov product from the right with
\[
(j^{TN}_{TS})^*\beta_N \in KK_G(C_0(TS), C_0(TS)).
\]
\end{lemma}
\begin{proof}
For all $x \in KK_G(A, C_0(TS))$, functoriality of the Kasparov product implies that
\[
\begin{split}
(j^{TM}_{TS})^* \circ(j^{TM}_{TS})_{!} (x) &= (j^{TM}_{TS})^* \circ (k^{TM}_{TN})_* (x \otimes_{C_0(TS)} \beta_N) \\
	&= x \otimes_{C_0(TS)} ((j^{TM}_{TS})^* \circ (k^{TM}_{TN})_* \beta_N).
\end{split}
\]
Since $(j^{TM}_{TS})^* \circ (k^{TM}_{TN})_* = (j^{TN}_{TS})^*$, the claim follows.
\end{proof}

\begin{lemma} \label{lem gysin index}
For any $G$-invariant closed subset $X\subset M$, and any $G$-invariant neighbourhood $V$ of $X$, the following diagram commutes:
\[
\xymatrix{
KK_G(X, TS) \ar[rr]^-{\relbar \otimes_{TS} [D_{TS}]} \ar[d]^-{(j_X^V)_*}& & KK_G(X, \pt) \ar[dd]^-{(j_X^V)_*} \\
KK_G(V, TS) \ar[d]^-{(j^{TM}_{TS})_{!} } & \\
KK_G(V, TM) \ar[rr]^-{\relbar \otimes_{TM} [D_{TM}]} & & KK_G(V, \pt). 
}
\]
\end{lemma}
\begin{proof}
For all $a \in KK_G(X, TS)$, functoriality and associativity of the Kasparov product imply that
\[
\begin{split}
\bigl( (j^{TM}_{TS})_{!} \circ (j_X^V)_* (a) \bigr) \otimes_{TM} [D_{TM}] &= \bigl(k^{TM}_{TN})_* ((j_X^V)_*(a) \otimes_{TS} \beta_N \bigr) \otimes_{TM} [D_{TM}] \\
	&= (j_X^V)_*(a) \otimes_{TS} \bigl((k^{TM}_{TN})_*(\beta_N) \otimes_{TM} [D_{TM}] \bigr).
\end{split}
\]
Now $(k^{TM}_{TN})^*[D_{TM}] = [D_{TN}]$, so
\[
\begin{split}
(k^{TM}_{TN})_*(\beta_N) \otimes_{TM} [D_{TM}] &=\beta_N \otimes_{TN} (k^{TM}_{TN})^*[D_{TM}]  \\
	&= \beta_N \otimes_{TN} [D_{TN}] \\
	&= [D_{TS}],
\end{split}
\]
where the last equality was proved in Proposition~\ref{lem:pairing.B.D}.
\end{proof}

%

\subsection{Localisation and Gysin maps} \label{sec Gysin loc}

Let $U$ and $V$ be as in Theorem \ref{thm loc 2}.
Consider the diagram
\begin{equation} \label{eq diag loc 1}
\xymatrix{
KK_G(M, TM) \ar[rr]^-{\ind_t} \ar[d]_-{(k^M_V)^*}& & KK_G(M, \pt) \ar[d]^-{(k^M_V)^*} \\
KK_G(V, TM) \ar[rr]^-{\relbar \otimes_{TM }[D_{TM}]}  \ar[d]_-{(j^{TM}_{TM^g})^*} & & KK_G(V, \pt) \\
KK_G(V, TM^g)    & &  \\
KK_G(V, TM^g) \ar[u]^-{\relbar \otimes_{TM^g} (j^{TN}_{TM^g})^*\beta_N}  \ar@<-18pt>@/_2pc/[uu]_-{(j^{TM}_{TM^g})_!}&  & \\
KK_G(\Ubar, TM^g) \ar[rr]^-{\relbar \otimes_{TM^g }[D_{TM^g}]} \ar[u]^-{(j^V_{\Ubar})_*}& & KK_G(\Ubar, \pt). \ar[uuu]_-{(j^V_{\Ubar})_*} 
}
\end{equation}
The top part of this diagram commutes because of functoriality of the Kasparov product. The part with the product with $(j^{TM}_{TM^g})^*\beta_N$ in it commutes by Lemma \ref{lem pull push}, applied with $A = C_0(V)$, and $S$ running over the connected components of $M^g$. 
The remaining part of the diagram commutes by Lemma \ref{lem gysin index}, applied in a similar way with $S$ a connected component of $M^g$, and  $X = \Ubar$.

Diagram \eqref{eq diag loc 1} can be extended as follows.
\begin{equation} \label{eq diag loc 2}
\hspace{-1.2cm}
\xymatrix{
 & & KK_G(M, TM) \ar[rr]^-{\ind_t} \ar[d]_-{(k^M_V)^*}& & KK_G(M, \pt) \ar[d]^-{(k^M_V)^*} \\
 KK_G(\pt, TM)  \ar[d]_-{(j^{TM}_{TM^g})^*} & KK_G(\Ubar, TM)  \ar[d]_-{(j^{TM}_{TM^g})^*}   \ar[l]_-{p^{\Ubar}_*} \ar[r]^-{(j^V_{\Ubar})_*} & KK_G(V, TM) \ar[rr]^-{\relbar \otimes_{TM }[D_{TM}]} \ar[d]_-{(j^{TM}_{TM^g})^*}  & & KK_G(V, \pt) \\
%
%
KK_G(\pt, TM^g)   & KK_G(\Ubar, TM^g) \ar[l]_-{p^{\Ubar}_*} \ar[r]^-{(j^V_{\Ubar})_*} & KK_G(V, TM^g)   & &  \\
%
%
 KK_G(\pt, TM^g)   \ar[u]^-{\relbar \otimes_{TM^g} (j^{TN}_{TM^g})^*\beta_N}  \ar[d]_-{\relbar \otimes_{TM^g }[D_{TM^g}]}  & KK_G(\Ubar, TM^g) \ar[u]^-{\relbar \otimes_{TM^g} (j^{TN}_{TM^g})^*\beta_N}  \ar[l]_-{p^{\Ubar}_*} \ar[r]^-{(j^V_{\Ubar})_*}  \ar[d]_-{\relbar \otimes_{TM^g }[D_{TM^g}]}  & KK_G(V, TM^g)  \ar[u]^-{\relbar \otimes_{TM^g} (j^{TN}_{TM^g})^*\beta_N} \ar@<-18pt>@/_2pc/[uu]_-{(j^{TM}_{TM^g})_!}&  & \\
%
KK_G(\pt, \pt) & KK_G(\Ubar, \pt) \ar[l]^-{p^{\Ubar}_*}  \ar@<0pt>@/_4pc/[rrruuu]_{{(j^V_{\Ubar})_*}}& & &
}
\end{equation}

The right hand part of this diagram is diagram \eqref{eq diag loc 1}, and hence commutes. The other parts commute by functoriality of $KK$-theory and the Kasparov product.

Theorem \ref{thm loc 2} implies that the maps $(j^V_{\Ubar})_*$ become invertible after localisation at $g$. We will also use inverses of the localised classes
\begin{equation}\label{eq beta loc}
\bigl((j^{TN}_{TM^g})^*\beta_N\bigr)_g \in KK_G(TM^g, TM^g)_g.
\end{equation}
\begin{lemma}\label{lem beta invertible}
The element \eqref{eq beta loc} is invertible.
\end{lemma}
\begin{proof}
By Lemma \ref{lem res Bott}, we have
\[
(j^{TN}_{TM^g})^*\beta_N = \bigl[\tau_{M^g}^*\Bigwedge N_{\C}\bigr].
\]
Atiyah and Segal showed in {\cite[Lemma~2.7]{AS}} that $\bigl[\Bigwedge N_{\C}\bigr]$ is invertible in $K_G^0(M^g)_g$.
The map
\[
\tau^*_{M^g}\colon K_G^0(M^g)\to KK_G(TM^g, TM^g) 
\]
sending a class  $[E] \in K_G^0(M^g)$ to $[\tau^*_{M^g}E]$, is a unital ring homomorphism. Hence so is its localisation at $g$. Therefore, the class
\[
\bigl[\tau_{M^g}^*\Bigwedge N_{\C}\bigr]_g = (\tau_{M^g}^*)_g \bigl[\Bigwedge N_{\C}\bigr]_g\quad \in KK_G(TM^g, TM^g)_g
\]
is invertible.
\end{proof}

\subsection{Evaluation}\label{sec eval}

Let $X$ and $Y$ be locally compact Hausdorff spaces with \emph{trivial} actions by a compact group $G$. Then
 the exterior Kasparov product
\[
KK(X, Y) \times KK_G(\pt, \pt) \to KK_G(X, Y)
\]
defines an isomorphism 
\beq{eq KK G triv}
KK(X,Y) \otimes R(G)\cong KK_G(X,Y).
\eeq

If $X$ is a point, this is a classical fact. We will also apply this isomorphism to the class $[D_{TM^g}] \in KK_G(TM^g, \pt)$. There it is trivial, since $G$ acts trivially on the Hilbert space in question.   
 In the only other case we will use the isomorphism \eqref{eq KK G triv}, we have $X = Y$, and this space has finitely many connected components. (To be precise, we will have $X = Y = TM^g$.)
Let us work out the isomorphism  explicitly in that case, for the cycles we will apply it to. These are $G$-equivariant Kasparov $(C_0(X), C_0(X))$-cycles of the form 
$(\Gamma_0(E), F, \pi)$, where $E\rightarrow X$ is a vector bundle (of finite rank). Let $a \in KK_G(X, X)$ be the class of a cycle of this form, and let $b\in KK(X, X)$ be the class defined by the same cycle, where the group action is ignored.
As $G$ acts trivially on $X$, each fibre of $E$  is a representation space of $G$. Suppose for simplicity that $X$ is connected; the general case follows by applying the arguments to its connected components. (This works since there are finitely many of them.)
Since $X$ is connected, the representations by $G$ on all fibres of $E$ are equivalent. Let $V$ be any one of these fibres, viewed as a representation space of $G$.
Denote by $1_G$ the ring identity of $R(G)$, i.e.\ the trivial representation of $G$.
Let 
 $E_0:=X\times V\rightarrow X$ be the trivial bundle with fibre $V$. 
 Consider the representations
 \[
 \begin{split}
 \pi_X^X\colon &C_0(X)\to \cB(C_0(X)); \\
 \pi_X^{E_0}\colon &C_0(X)\to \cB(\Gamma_0(E_0))
 \end{split}
 \]
 defined by pointwise multiplication. 
Then 
\beq{eq:rotation.KK}
\bigl(\bigl[C_0(X), 0, \pi_X^X\bigr]\otimes[V] \bigr) + (b\otimes 1_G)=\bigl( \bigl[\Gamma_0(E_0), 0, \pi_X^{E_0}\bigr]\otimes 1_G\bigr) + a \quad \in KK_G(X, X).
\eeq
In fact, both sides of~(\ref{eq:rotation.KK}) are represented by the cycle
\beq{eq:product.diff.action}
\bigl(\Gamma_0(E_0\oplus E), 0\oplus F, \pi_X^{E_0}\oplus\pi \bigr),
\eeq
but, initially,  with \emph{different} $G$-actions. Namely, for the left-hand side of~(\ref{eq:rotation.KK}), $G$ acts on the first summand $E_0$ in~(\ref{eq:product.diff.action}), while for the right-hand side of~(\ref{eq:rotation.KK}), $G$ acts on the second summand $E$ in~(\ref{eq:product.diff.action}). 
As $G$ acts trivially on $X$, representations of $G$ commute with those of $C_0(X)$. Since, in addition, $F$ is $G$-invariant, these two actions by  $G$  can be connected by a rotation homotopy, so \eqref{eq:rotation.KK} follows. 
In that equality, $a$ is represented as an element of $KK(X, Y)\otimes R(G).$

In general, using \eqref{eq KK G triv}, 
one can apply the evaluation $\ev_g = 1\otimes \ev_g$ as a map
\begin{equation} \label{eq ev KK}
\ev_g\colon KK_G(X,Y) \to KK(X,Y)\otimes \C.
\end{equation}
This map is compatible with localisation at $g$, in the sense that the following diagram commutes:
\[
\xymatrix{
KK_G(X,Y) \ar[r]^{\ev_g} \ar[d]& KK(X,Y) \otimes \C. \\
KK_G(X,Y)_g \ar[ur]_-{(\ev_g)_g} & 
}
\]
If $a\in KK_G(X,Y)$, we will also write 
\[
a(g) := \ev_g(a) \quad \in KK(X, Y) \otimes \C.
\]

The evaluation map \eqref{eq ev KK} is compatible with Kasparov products. This follows from the facts that the isomorphism \eqref{eq KK G triv} is compatible with the product, that Kasparov products in $R(G)$ coincide with tensor products of representations, and that the character of the tensor product of two finite-dimensional representations is the product of the characters of the individual representations.

Hence we can attach the following commutative diagram to the lower left hand side of \eqref{eq diag loc 2}:
\begin{equation} \label{eq diag loc 2a}
\xymatrix{
KK(\pt, TM^g)\otimes \C  & KK_G(\pt, TM^g) \ar[l]_-{\ev_g}    \\
KK(\pt, TM^g)\otimes \C 
\ar[d]_-{(\relbar \otimes_{TM^g } [D_{TM^g}])\otimes 1} 
\ar[u]^-{\bigl(\relbar \otimes_{TM^g}(j^{TM}_{TM^g})^*\beta_N (g) \bigr)\otimes 1}    & KK_G(\pt, TM^g)   \ar[u]^-{\relbar \otimes_{TM^g} (j^{TN}_{TM^g})^*\beta_N}  \ar[d]_-{\relbar \otimes_{TM^g }[D_{TM^g}]} \ar[l]_-{\ev_g}    \\
\C & KK_G(\pt, \pt) \ar[l]^-{\ev_g}  
}
\end{equation}
Here, $[D_{TM^g}]\in KK(TM^g, \pt)$ is identified with $[D_{TM^g}]\otimes 1\in KK(TM^g, \pt) \otimes R(G)$, so that $\ev_g([D_{TM^g}])=[D_{TM^g}]\otimes 1.$  In particular, when $M^g=\pt$, the vertical map on the lower left corner is the identity.

By Lemma \ref{lem ind top cpt} and compactness of $M^g$, the map 
\[
\relbar \otimes_{TM^g } [D_{TM^g}]\colon KK(\pt, TM^g)\to KK(\pt, \pt)
\]
is the Atiyah--Singer topological index map $\ind_t^{\AS}$.
We will use the same notation for  its extension to a map $KK(\pt, TM^g) \otimes \C \to \C$.

Using commutativity of \eqref{eq diag loc 2} and \eqref{eq diag loc 2a}, and invertibility of the localised maps
$((j^V_{\Ubar})_*)_g$
 and classes \eqref{eq beta loc},  we obtain the commutative diagram
\begin{equation} \label{eq diag loc 3}
\hspace{-2.5cm}
\xymatrix{
& KK_G(M, TM)_g \ar[rrrr]^-{(\ind_t)_g} \ar[d]^-{(p^{\Ubar}_*)_g \circ ((j^V_{\Ubar})_*)_g^{-1} \circ (k^M_V)^*_g} & & & &
KK_G(M, \pt)_g\ar[ddd]_-{((j^V_{\Ubar})_*)_g^{-1} \circ (k^M_V)^*_g} \\
%
& KK_G(\pt, TM)_g 
  \ar[d]^-{(j^{TM}_{TM^g})^*} && & &\\
 KK(\pt, TM^g)\otimes \C  
 \ar[d]_-{\ind_t^{\AS} \bigl( \relbar \otimes_{TM^g} ((j^{TN}_{TM^g})^*\beta_N)^{-1}(g)\bigr)}
 & KK_G(\pt, TM^g)_g \ar[l]_-{(\ev_g)_g}  \ar[d]^-{ \relbar \otimes_{TM^g} ((j^{TN}_{TM^g})^*\beta_N)_g^{-1} \otimes_{TM^g}[D_{TM^g}]_g}   & & & \\
\C & KK_G(\pt, \pt)_g  \ar[l]_-{(\ev_g)_g}  &&& & KK_G(\Ubar, \pt)_g. \ar[llll]^-{(p^{\Ubar}_*)_g}& 
}
\end{equation}

\subsection{The $g$-symbol class} \label{sec g symbol}

Recall that in~(\ref{eq sigma D}) we defined the class
\[
[\sigma_D] \in KK_G(M, TM).
\]
The last ingredient of the proof of Theorem \ref{thm fixed pt} is
 a class defined by $\sigma_D$ in the topological $K$-theory of $TM$, localised at $g$. In Section \ref{sec non-loc}, we will describe this class more explicitly, and use it to obtain another expression for the $g$-index.
\begin{definition}\label{def g symbol}
The \emph{$g$-symbol class} class of $D$ is the class
${\sigDg}$ in the localised topological $K$-theory of $TM$ defined by
\begin{equation}\label{eq:sigma}
{\sigDg} := (p^{\Ubar}_*)_g \circ ((j^V_{\Ubar})_*)_g^{-1} \circ (k^M_V)^*_g[\sigma_D]_g \quad \in KK_G(\pt, TM)_g.
\end{equation}
\end{definition}

The $g$-symbol class generalises the usual symbol class in the compact case.
\begin{lemma} \label{lem sigma cpt}
If $M$ is compact, then 
${\sigDg}$ is the localisation at $g$ of the usual class of $\sigma_D$ in $KK_G(\pt, TM)$.
\end{lemma}
\begin{proof}
If $M$ is compact, then we can choose $U = V = M$. Then, since the map $p^M\colon M\to \pt$ is proper, we have 
\[
{\sigDg} = (p^M_*[\sigma_D])_g,
\]
which is the usual symbol class by Lemma \ref{lem norm sigma D}.
%
%
\end{proof}

We now prove some properties of the $g$-symbol class that will be used in the proof of Theorem \ref{thm fixed pt}.
As before, we write
$
\tilde \sigma_D := \frac{\sigma_D}{\sqrt{\sigma_D^2+1}}.
$
\begin{lemma} \label{lem sigma D res V}
The class
\[
(k^M_V)^*_g[\sigma_D]_g \in KK_G(V, TM)_g
\]
is the localisation at $g$ of the class
\[
[\sigma_D|_V]_{TM} :=  \bigl[\Gamma_0(\tau_{V}^*(E|_V)), \tilde \sigma_D|_{TV}, \pi_V\bigr] \quad \in KK_G(V, TM).
\]
Here the $C_0(TM)$ valued inner product on $\Gamma_0(E|_V)$ is defined by the natural $C_0(TV)$-valued inner product, composed with the inclusion $k^{TM}_{TV}$.
\end{lemma}
\begin{proof}
The class
\[
(k^M_V)^*[\sigma_D]\in KK_G(V, TM)
\]
is represented by the Kasparov cycle
\begin{multline*}
\bigl(\Gamma_0(\tau_{M}^*E), \tilde \sigma_D, (k^M_V)^*\pi_M\bigr) = \bigl(\Gamma_0(\tau_{V}^*(E|_{V}), \tilde \sigma_D|_{TV}, \pi_V \bigr) \\ \oplus \bigl(\Gamma_0(\tau_{M\setminus V}^*(E|_{M\setminus V})), \tilde \sigma_D|_{TM\setminus TV}, 0\bigr).
\end{multline*}
The second term on the right hand side is a degenerate cycle, so the claim follows.
\end{proof}

Consider
the class
\[
_{\Ubar}[\sigma_D|_{TM^g}] := \bigl[\Gamma_0(\tau_{M^g}^* (E|_{M^g})), \tilde \sigma_D|_{TM^g}, (j^{\Ubar}_{M^g})_*\pi_{M_g} \bigr] \in KK_G(\Ubar, TM^g).
\]

\begin{lemma} \label{lem res sigma}
We have
\[
(j^V_{\Ubar})_* \bigl({_{\Ubar}}[\sigma_D|_{TM^g}] \bigr)= (j^{TM}_{TM^g})^*[\sigma_D|_V]_{TM} \quad \in KK_G(V, TM^g).
\]
\end{lemma}
\begin{proof}
By definition,
\[
(j^{TM}_{TM^g})^*[\sigma_D|_V]_{TM} = \bigl[  \Gamma_0(\tau_V^*(E|_V)) \otimes_{j^{TM}_{TM^g}} C_0(TM^g), \tilde \sigma_D|_V\otimes 1, \pi_V \otimes 1  \bigr].
\]
The map 
\[
 \Gamma_0(\tau_V^*(E|_V)) \otimes_{j^{TM}_{TM^g}} C_0(TM^g) \to \Gamma_0(\tau_{M^g}^*(E|_{M^g}))
\]
that maps $s\otimes \varphi$ to $\varphi s|_{TM^g}$, for $s\in  \Gamma_0(\tau_V^*(E|_V))$ and $\varphi \in C_0(TM^g)$, is an isomorphism of Hilbert $C_0(TM^g)$-modules. It intertwines the operators $\tilde \sigma_D|_V\otimes 1$ and $\tilde \sigma_D|_{TM^g}$, and the representations $\pi_V\otimes 1$ and 
\[
(j^{V}_{M^g})_*\pi_{M^g} =(j^V_{\Ubar})_* (j^{\Ubar}_{M^g})_*\pi_{M^g}.
\]
The lemma is then proved.
\end{proof}

\begin{proposition}
\label{prop:symbol.M^g}
The class
\[
(j^{TM}_{TM^g})^*_g {\sigDg}  \in KK_G(\pt, TM^g)_g
\]
is the localisation at $g$ of the usual class
$
 [\sigma_D|_{TM^g}]
$
in the equivariant topological $K$-theory of $TM^g$.
\end{proposition}
\begin{proof}
By commutativity of (the top left part of) diagram  \eqref{eq diag loc 2}, we have
\[
(j^{TM}_{TM^g})^*_g {\sigDg} = (p^{\Ubar}_*)_g \circ ((j^V_{\Ubar})_*)_g^{-1}\circ (j^{TM}_{TM^g})^*_g 
\circ (k^M_V)^*_g [\sigma_D]_g.
\]
By Lemma \ref{lem sigma D res V}, we have
\[
(k^M_V)^*_g [\sigma_D]_g  = ([\sigma_D|_V]_{TM})_g. 
\]
By Lemma \ref{lem res sigma} we have
\[
((j^V_{\Ubar})_*)_g^{-1}\circ (j^{TM}_{TM^g})^*_g  ([\sigma_D|_V]_{TM})_g = _{\Ubar}[\sigma_D|_{TM^g}]_g. 
\]
By Lemma \ref{lem norm sigma D}, we have
\[
p^{\Ubar}_* \bigl({}_{\Ubar}[\sigma_D|_{TM^g}] \bigr)= [\sigma_D|_{TM^g}] \quad \in KK_G(\pt, TM^g).
\]
So the claim follows.
\end{proof}

We have now finished all preparation needed to prove Theorem \ref{thm fixed pt}.

\medskip
\noindent
\emph{Proof of Theorem \ref{thm fixed pt}.}
Using  Kasparov's index theorem, Theorem \ref{thm index}, and commutativity of \eqref{eq diag loc 3}, we find that
\[
\begin{split}
\ind_g(D) &=  (\ev_g)_g \circ (p^{\Ubar}_*)_g \circ ((j^V_{\Ubar})_*)_g^{-1} \circ (k^M_V)^*_g[D] \\
	&= (\ev_g)_g \circ (p^{\Ubar}_*)_g \circ ((j^V_{\Ubar})_*)_g^{-1} \circ (k^M_V)^*_g \circ  (\ind_t)_g[\sigma_D]_g \\
	&=  \ind_t^{\AS} \bigl( \bigl((j^{TM}_{TM^g})^*\sigma^D_g\bigr)(g) \otimes_{TM^g} \bigl(    (j^{TN}_{TM^g})^*\beta_N\bigr)^{-1}(g)\bigr).
\end{split}
\]
By Lemma \ref{lem res Bott} and Proposition~\ref{prop:symbol.M^g}, the latter expression equals
\[
\ind_t^{\AS} \bigl( [\sigma_{D}|_{TM^g}](g) \otimes_{TM^g} \bigl[\Bigwedge N_{\C}\bigr]^{-1}(g) \bigr).
\]
Furthermore,
\[
[\sigma_D|_{TM^g}](g) \otimes_{TM^g}  \bigl[\tau_{M^g}^*\Bigwedge N_{\C}\bigr]^{-1}(g)=[\sigma_D|_{TM^g}](g)\cdot\bigl[\Bigwedge N_{\C}\bigr]^{-1}(g),
\]
where the dot means the right $K^0_G(M^g)$-module structure of $K^0_G(TM^g)$. We conclude that
\[
\ind_g(D) = \ind_t^{\AS} \bigl( [\sigma_{D}|_{TM^g}](g) \cdot \bigl[\Bigwedge N_{\C}\bigr]^{-1}(g) \bigr).
\]
Theorem \ref{thm fixed pt} now follows from the definition of the topological index map~(\ref{eq:AS.top.index}), and multiplicativity of the Chern character.
\hfill $\square$

\subsection{The index pairing} \label{sec index pair pf}

The arguments used to prove Theorem \ref{thm fixed pt} also imply Theorem \ref{thm index pair} about the index pairing. In fact, the parts of the proof of Theorem \ref{thm fixed pt} about localisation in the first entry of $KK$-theory are not needed in the proof Theorem \ref{thm index pair}.

The key step is a localisation property of the $K$-homology class of $D$, localised at $g$.
\begin{proposition}\label{prop loc Khom D}
We have
\[
[D]_g = (j^{TM}_{TM^g})^*_g[\sigma_D]_g \otimes _{TM^g} [\tau_{M^g}^*\Bigwedge N_{\C}]_g^{-1} \otimes_{TM^g}[D_{TM^g}]_g \quad \in KK_G(M, \pt)_g.
\]
\end{proposition}
\begin{proof}
Lemmas \ref{lem pull push} and \ref{lem gysin index} imply that the following diagram commutes:
\[
\xymatrix{
KK_G(M, TM) \ar[rr]^-{\relbar \otimes_{TM} [D_{TM}]} \ar[d]_-{(j^{TM}_{TM^g})^*} & & KK_G(M, \pt) \\
KK_G(M, TM^g) & & KK_G(M, TM^g). \ar[ull]_-{(j^{TM}_{TM^g})_!} \ar[u]_-{\relbar \otimes_{TM^g} [D_{TM^g}]}
 \ar[ll]^-{\relbar \otimes_{TM^g} (j^{TN}_{TM^g})^*\beta_N} 
}
\]
Therefore, the claim follows from Lemmas \ref{lem res Bott} and \ref{lem beta invertible}, and Theorem \ref{thm index}.
\end{proof}

\noindent\emph{Proof of Theorem \ref{thm index pair}.}
Let  $[F] \in KK_G(\pt, M)$ be as in Subsection \ref{sec index pair}. By compatibility of the Kasparov product with localisation and evaluation, Proposition \ref{prop loc Khom D} implies that
\[
\begin{split}
([F]\otimes_M[D])(g) &= ([F]_g\otimes_M[D]_g)(g) \\
	&= ([F]_g\otimes_M (j^{TM}_{TM^g})_g^*[\sigma_D]_g)(g) \otimes _{TM^g} [\tau_{M^g}^*\Bigwedge N_{\C}](g)^{-1} \otimes_{TM^g}[D_{TM^g}](g).
\end{split}
\]
Now
\[
 ([F]_g\otimes_M (j^{TM}_{TM^g})_g^*[\sigma_D]_g)(g) =  \bigl[\tau_{M^g}^*(F|_{M^g}) \bigr](g)\otimes  [\sigma_D|_{TM^g}](g) \quad \in KK(\pt, TM^g) \otimes \C,
\]
where on the right hand side, the tensor product denotes the ring structure on the topological $K$-theory of $TM^g$. Therefore, and because
 $[D_{TM^g}](g) = [D_{TM^g}] \otimes 1 \in KK(TM^g, \pt) \otimes \C$,
the claim follows from Lemma \ref{lem ind top cpt}. 
\hfill $\square$


\section{Examples and applications}\label{sec examples}

The $g$-index was defined in terms of $KK$-theory, but Theorem \ref{thm fixed pt} allows us to express it entirely in cohomological terms. Using this theorem, we can compute the $g$-index explicitly in examples, and show how it is related to other indices.

For finite fixed point sets, Theorem \ref{thm fixed pt} has a simpler form, as discussed in Subsection \ref{sec finite}.
In Subsection \ref{sec hol lin}, we give a linearisation theorem for the $g$-index of a twisted Dolbeault--Dirac operator on a complex manifold, in the case of a finite fixed point set.
 We then work out the example of the Dolbeault--Dirac operator on the complex plane, acted on by the circle, in Subsection \ref{sec:S^1.on.R^2}.
 An illustration of the linearisation theorem is given in Subsection \ref{sec T S2}, where we apply it to the two-sphere, to decompose the usual equivariant index.
  In Subsection \ref{sec ds}, we realise characters of discrete series representations of semisimple Lie groups on regular points of a maximal torus, in terms of the $g$-index. 
 For Fredholm operators, and in particular Callias-type deformations of Dirac operators, we describe the relation between the $g$-index and the character of the action by $g$ on the kernel of such an operator, in Subsection \ref{sec Fredholm}. 
  We then give a relation with an index studied by Braverman in Subsection \ref{sec Braverman}, and a relative index theorem along the lines of work by Gromov and Lawson in Subsection \ref{sec rel index}. In Subsection \ref{sec Hodge Spin}, we mention some geometric consequences of the vanishing or nonvanishing of the $g$-index of a Hodge-Dirac or $\Spin$-Dirac operator.

\subsection{Finite fixed point sets} \label{sec finite}

If the fixed point set $M^g$ is $0$-dimensional, then $TM^g=M^g$, $\tau_{M^g}$ is the identity map, $\Todd(TM^g\otimes\C)$ is trivial and 
\[
\ch \bigl( [\sigma_D|_{TM^g}](g) \bigr)=\Tr(g|_{E^+})-\Tr(g|_{E^-}).
\] 
Furthermore, since $M^g$ only consists of  isolated points, we have
\[
K^0(M^g)=\bigoplus_{m\in M^g}\Z = H^*(M^g),
\]
and the Chern character is the identity map. So we now have, at a fixed point $m\in M^g$,
\[
\ch\bigl(\bigl[\Bigwedge N_{\C}\bigr](g) \bigr)_m
 =\ch\bigl(\bigl[\Bigwedge TM_{\C}|_{M^g}\bigr](g) \bigr)_m
= \det_{\R}(1-g|_{T_mM}).
\]
The last equality is obtained by evaluating the virtual character of $\Bigwedge T_mM_{\C}$ at $g$, so one obtains
\[
\Tr_{\C}(g|_{\bigwedge^{\even}T_mM_{\C}})- \Tr_{\C}(g|_{\bigwedge^{\odd}T_mM_{\C}}).
\]
Therefore, Theorem \ref{thm fixed pt} implies the following generalisation of Atiyah--Bott's fixed point theorem \cite[Theorem A]{AB2} to noncompact manifolds, but for compact~$G$. 
\begin{corollary}
When $M^g$ is a finite set of points, 
\begin{equation} \label{eq finite fixed pt}
\ind_g(D)=\sum_{m\in M^g}\frac{\Tr(g|_{E_m^+}) - \Tr(g|_{E_m^-})}{\det_{\R}(1-g^{-1}|_{T_mM})}.
\end{equation}
\end{corollary}

\begin{remark}
In the statement of Atiyah--Bott fixed point theorem, the denominator is $|\det_{\R}(1-g|_{T_{m}M})|$. In our case, $g$ is contained in a compact group $G$, so the real eigenvalues of $g$ are $1$ or $-1$. Thus $\det_{\R}(1-g^{-1}|_{T_{m}M})$ is always positive. See also page 186 in~\cite{BGV}. Also, the fact that $g$ acts orthogonally on $T_{m}M$ implies that $\det_{\R}(1-g^{-1}|_{T_{m}M}) = \det_{\R}(1-g|_{T_{m}M})$
\end{remark}

Now suppose $M$ is a complex manifold, and suppose $g$ is holomorphic. Let $F\to M$ be a holomorphic vector bundle, and consider the Dolbeault--Dirac operator $\bar \partial_F + \bar\partial_F^*$ on $M$, coupled to $F$.
\begin{corollary} \label{cor fin fixed pt hol}
If $M^g$ is a finite set of points, then
\begin{equation} \label{eq fin fixed pt hol}
\ind_g(\bar \partial_F + \bar\partial_F^*)=\sum_{m\in M^g}\frac{\Tr_{\C}(g|_{F_m})}{\det_{\C}(1-g^{-1}|_{T_{m}M})}.
\end{equation}
\end{corollary}
For equivalent expressions, note that 
\[
{\det}_{\C}(1-g^{-1}|_{T^{1, 0}_{m}M}) = {\det}_{\C}(1-g^{-1}|_{T_{m}M}) = {\det}_{\C}(1-g|_{T^{0,1}_{m}M}) 
\]
in \eqref{eq fin fixed pt hol}.
\begin{proof}
In Theorem~4.12 of~\cite{AB2}, it is shown that in this situation, the right hand side of \eqref{eq finite fixed pt} equals the right hand side of \eqref{eq fin fixed pt hol}. 
The key observation is that the supertrace of $g|_{\bigwedge^*(T^{0, 1}M)}$ is cancelled by the second factor in 
\[
{\det}_{\R}(1-g^{-1}|_{T_{m}M})={\det}_{\C}(1-g^{-1}|_{T_{m}^{1, 0}M}){\det}_{\C}(1-g^{-1}|_{T_{m}^{0, 1}M}).
\]
(See also {\cite[Corollary~6.8]{BGV}}.)
\end{proof}

\subsection{A holomorphic linearisation theorem} \label{sec hol lin}

A tool used in some index problems is a \emph{linearisation theorem}, relating an index to indices on vector spaces. See for example Chapter~4 of \cite{GGK} and Theorem~7.2 in \cite{Braverman}. A version for Callias-type operators can be decuced from Theorem 2.16 in \cite{BrShi}. In those references, cobordism arguments are used to prove linearisation theorems. We will use the excision property of the $g$-index to obtain an analogous result. (So we do not use Theorem~\ref{thm fixed pt} here.) We will state and prove this result in the setting of Corollary~\ref{cor fin fixed pt hol}, where $M$ is a complex manifold, $D$ is the Dolbeault--Dirac operator coupled to a holomorphic vector bundle $F\to M$, and $M^g$ is finite. A more general statement, where $M^g$ is not finite or $D$ is not a Dolbeault--Dirac operator, is possible, but would be less explicit.

Under these assumptions, for any $m\in M^g$, let $\bar\partial^{T_mM}$ be the Dolbeault operator on the complex vector space $T_mM$.
\begin{corollary}[Holomorphic linearisation theorem]\label{cor lin}
We have
\[
\ind_g(\bar \partial_F + \bar\partial_F^*)=\sum_{m\in M^g} {\Tr_{\C}(g|_{F_m})} \ind_g \bigl( \bar\partial^{T_mM} + (\bar\partial^{T_mM})^*\bigr).
\]
\end{corollary}
\begin{proof}
By Lemma \ref{lem excision}, the $g$-index of $\bar \partial_F + \bar\partial_F^*$ equals the $g$-index of the Dolbeault--Dirac operator on the union over $m\in M^g$ of the tangent spaces $T_mM$, coupled to the vector bundle which on every space $T_mM$ is trivial with fibre $F_m$. It follows directly from the definition that the $g$-index is additive with respect to disjoint unions. Hence
\iftoggle{long}
{
\[
\ind_g(\bar \partial_F + \bar\partial_F^*)=\sum_{m\in M^g} \ind_g \bigl( \bar\partial^{T_mM}\otimes 1_{F_m} + (\bar\partial^{T_mM})^*\otimes 1_{F_m}\bigr).
\]

Now, in terms of the $R(G)$-module structure of $KK_G(M, \pt)$, we have for all $m \in M^g$,
\[
\bigl[ \bar\partial^{T_mM}\otimes 1_{F_m} + (\bar\partial^{T_mM})^*\otimes 1_{F_m}] = \bigl[ \bar\partial^{T_mM}  + (\bar\partial^{T_mM})^* \bigr] \otimes [F_m]\quad \in KK_G(M, \pt).
\]
It follows that 
\begin{multline*}
\ind_g \bigl( \bar\partial^{T_mM}\otimes 1_{F_m} + (\bar\partial^{T_mM})^*\otimes 1_{F_m}\bigr) \\=  {\Tr_{\C}(g|_{F_m})} \ind_g \bigl( \bar\partial^{T_mM} + (\bar\partial^{T_mM})^*\bigr).
\end{multline*}
}
{
\[
\begin{split}
\ind_g(\bar \partial_F + \bar\partial_F^*)&=\sum_{m\in M^g} \ind_g \bigl( \bar\partial^{T_mM}\otimes 1_{F_m} + (\bar\partial^{T_mM})^*\otimes 1_{F_m}\bigr) \\
&=\sum_{m\in M^g} {\Tr_{\C}(g|_{F_m})} \ind_g \bigl( \bar\partial^{T_mM} + (\bar\partial^{T_mM})^*\bigr).
\end{split}
\]
}
\end{proof}

An example on computing and explicitly realising an index of the form
\[
 \ind_g \bigl( \bar\partial^{T_mM} + (\bar\partial^{T_mM})^*\bigr)
\]
as in Corollary \ref{cor lin}, is given in the next subsection. An example showing that the linearisation theorem gives a natural result if $M$ is compact is given in Subsection \ref{sec T S2}.

\subsection{The circle acting on the plane}
\label{sec:S^1.on.R^2}

Consider the usual action by the circle $\bT^1 = \U(1)$ on the complex plane $\C$, and the (untwisted) Dolbeault--Dirac operator $\bar \partial + \bar\partial^*$ on $\C$. We will compute the distribution $\Theta$ on $\bT^1$ given by the function
\begin{equation} \label{eq dist ind}
g\mapsto \ind_g(\bar \partial + \bar\partial^*). 
\end{equation}
This function is defined on the set of elements $g\in \bT^1$ with dense powers, i.e.\ the elements of the form $g = e^{\ii\alpha}$, where $\alpha \in \R\setminus 2\pi \Q$. So the function is defined almost everywhere.

\iftoggle{long}
{
\begin{lemma} \label{lem circle plane}
The sum of functions
\begin{equation} \label{eq sum char}
 \sum_{k=0}^{\infty} \bigl( g\mapsto  g^{-k} \bigr).
\end{equation}
converges as a distribution on $\bT^1$ to  $\Theta$.
\end{lemma}
\begin{proof}
By Corollary \ref{cor fin fixed pt hol}, we have for $g = e^{i\alpha}$, with $\alpha\in \R\setminus 2\pi\Q$,
\[
\ind_g(\bar \partial + \bar\partial^*) = \frac{1}{1-g^{-1}}.
\]
So the function \eqref{eq dist ind} is given by $g\mapsto 1/(1-g^{-1})$ almost everywhere.

Note that for all $n\in \N$, and $g\in \bT^1$,
\[
\frac{1}{1-g^{-1}} - \sum_{k=0}^{n} g^{-k} = \frac{g^{-(n+1)}}{1-g^{-1}}.
\]
For any $\varphi \in C^{\infty}(\bT^1)$, the oscillatory integral
\[
\int_{\bT^1} \varphi(g)\frac{g^{-(n+1)}}{1-g^{-1}} \, dg
\]
tends to zero as $n\to \infty$. The claim follows.
\end{proof}
}
{
By Corollary \ref{cor fin fixed pt hol}, we have for 
\iftoggle{long}{
$g = e^{\ii\alpha}$, with $\alpha\in \R\setminus 2\pi\Q$,
}
{
such $g$,
}
\[
\ind_g(\bar \partial + \bar\partial^*) = \frac{1}{1-g^{-1}}.
\]
So the function \eqref{eq dist ind} is given by $g\mapsto 1/(1-g^{-1})$ almost everywhere. One can deduce that the sum of functions
\begin{equation} \label{eq sum char}
 \sum_{k=0}^{\infty} \bigl( g\mapsto  g^{-k} \bigr)
\end{equation}
converges as a distribution on $\bT^1$ to  $\Theta$.
}

\iftoggle{long}
{
Lemma \ref{lem circle plane} 
}
{This}
allows us to describe the $g$-index of $\bar \partial + \bar \partial ^*$ in terms of its kernel.
Indeed, consider the Euclidean density $dz = dx\, dy$ on $\C$, and the corresponding space  $L^2(\C)$. Let $\cO(\C)$ be the space of holomorphic functions on $\C$.
Let $\psi \in C^{\infty}(\C)$ be a positive, $\bT^1$-invariant function. Let  ${L^2}(\C, \psi)$ be the completion  of $C^{\infty}_c(\C)$ to a Hilbert space with respect to the inner product
\begin{equation}
\label{eq:psi.inner.product}
(f_1, f_2)_{\psi} := (\psi f_1, \psi f_2)_{L^2(\C)}.
\end{equation}
Let $\pi$ be the representation of $\bT^1$ in ${L^2}(\C, \psi)$ given by
\[
(\pi(g)f)(z) = f(g^{-1}z),
\]
for all $g\in \bT^1$, $f\in {L^2}(\C, \psi)$ and $z\in \C$.

Set
\[
\cO_{L^2}(\C, \psi) := \cO(\C) \cap {L^2}(\C, \psi).
\]
For $k\in \Z_{\geq 0}$, let $e^k\in \cO(\C)$ be the function $z\mapsto z^k$. 
Then for all $k\in \Z_{\geq 0}$ and $z\in \C$, 
\beq{eq pi ek}
\pi(g)e^k = g^{-k}e^k.
\eeq
Suppose $\psi$ was chosen so that $e^k\in  {L^2}(\C, \psi)$ for all $k$.  For example, one can take $\psi(z) = e^{-|z|^2/2}$.

Let $\Omega^{0,*}_{L^2}(\C)$ be the Hilbert space of square-integrable forms of type $(0,*)$. Let  $\Omega^{0,*}_{L^2}(\C, \psi)$ be the analogous Hilbert space with the inner product weighted by $\psi$ as in \eqref{eq:psi.inner.product}.
Set
\[
\ker_{L^2, \psi} (\bar \partial + \bar\partial^*)^{\pm} := \ker (\bar \partial + \bar\partial^*)^{\pm} \cap  \Omega^{0,*}_{L^2}(\C, \psi).
\]
We can realise the distribution $\Theta$ given by the $g$-indices of $\bar \partial + \bar\partial^*$ in terms of the representation of $\bT^1$ in this space.
\begin{proposition} \label{prop C psi}
The restriction of the representation $\pi$ of $\bT^1$ to $\ker_{L^2, \psi} (\bar \partial + \bar\partial^*)^{\pm}$ has a distributional character $\chi^{\pm}$, and we have
\[
\Theta = \chi^+ - \chi^- \quad \in \cD'(\bT^1).
\]
\end{proposition}
\iftoggle{long}
{

In the proof of this proposition, we will use a natural Hilbert basis of the space $\cO_{L^2}(\C, \psi)$.
\begin{lemma}\label{lem basis}
The functions $\{e^k\}_{k\ge0}$ form an orthogonal basis of $\cO_{L^2}(\C, \psi)$.
\end{lemma}
\begin{proof}
For all $k, l \in \Z_{\geq 0}$ and $g\in \bT^1$, we have
\[
(e^k, e^l)_{\psi} = ( (gz)^k, (ge^l))_{\psi} = g^{l-k} (e^k, e^l)_{\psi}
\]
Since this holds for all $g\in \bT^1$, we find that
$\{e^k\}_{k\in \Z_{\geq 0}}$ is an orthogonal set in ${L^2}(\C, \psi)$.

Moreover, consider an holomorphic function $f\in \cO_{L^2}(\C, \psi)$. Then
\begin{equation}
\label{eq:norm.f}
\|f\|^2_{\psi}=\int_{\C}f(z)\overline{f(z)}\psi(z)^2 \, dz <\infty.
\end{equation}
The function $f$ can be represented by a power series $f(z)=\sum_{k=0}^{\infty}a_k z^k$ converging  for all $z\in\C$, hence uniformly converging on compact subsets of $\C$.
Let $\{K_n\}_{n\in\N}$ be an increasing sequence of compact subsets of $\C$ such that $\cup_{n\in \N} K_n=\C$.
Then 
we have
\begin{multline}
\label{eq:norm.f.series}
\|f\|^2_{\psi}=\lim_{n\to\infty}\int_{K_n}\lim_{N\to\infty}\Bigl(\sum_{k=0}^N a_kz^k \Bigr)\Bigl(\sum_{l=0}^{\infty}\overline{a_l}\bar z^l \Bigr)\psi(z)^2dz \\
=\lim_{N\to\infty}\sum_{k=0}^{N}|a_k|^2 \|e^k\|_{\psi}^2. 
\end{multline}
For $N\in \N$, consider the partial sum
 $f_N(z)=\sum_{k\le N}a_k z^k$. As $\{e^k\}_{k\ge 0}$ is an orthogonal set, by the Pythagorean theorem, and by~(\ref{eq:norm.f.series}),  we have
\[
\|f-f_N\|_{\psi}^2=\|f\|_{\psi}^2-\|f_N\|_{\psi}^2=\sum_{k=0}^{\infty}|a_k|^2 \|z_k\|_{\psi}^2-\sum_{k=0}^{N}|a_k|^2 \|z_k\|_{\psi}^2.
\]
This tends to zero as as $N\to\infty$, because of (\ref{eq:norm.f}) and \eqref{eq:norm.f.series}. 
Therefore, the span of $\{e^k\}_{k\ge0}$ is dense in  $\cO_{L^2}(\C, \psi)$. 
\end{proof}

\noindent \medskip \emph{Proof of Proposition \ref{prop C psi}.}
First note that 
\[
\begin{split}
\ker (\bar \partial + \bar\partial^*)^{+} &=  \cO(\C);\\
 \ker (\bar \partial + \bar\partial^*)^{-} &= 0.
\end{split}
\]
 so we only need to consider the even part of $\ker_{L^2, \psi} (\bar \partial + \bar\partial^*)$, which equals 
\beq{eq ker C plus}
\ker_{L^2, \psi} (\bar \partial + \bar\partial^*)^{+} =  \cO_{L^2}(\C, \psi).
\eeq
By Lemma \ref{lem basis}, the functions $\{e^k\}_{k\ge 0}$ form an orthogonal basis of $\cO_{L^2}(\C, \psi)$. By \eqref{eq pi ek}, the character of the representation $\pi$ on the space \eqref{eq ker C plus} equals the distribution \eqref{eq sum char}. 
So the claim follows from Lemma~\ref{lem circle plane}.
\hfill $\square$
}
{
\begin{proof}
First note that 
\[
\begin{split}
\ker (\bar \partial + \bar\partial^*)^{+} &=  \cO(\C);\\
 \ker (\bar \partial + \bar\partial^*)^{-} &= 0.
\end{split}
\]
So we only need to consider the even part of $\ker_{L^2, \psi} (\bar \partial + \bar\partial^*)$, which equals 
\beq{eq ker C plus}
\ker_{L^2, \psi} (\bar \partial + \bar\partial^*)^{+} =  \cO_{L^2}(\C, \psi).
\eeq
The functions $\{e^k\}_{k\ge 0}$ form an orthogonal basis of $\cO_{L^2}(\C, \psi)$. By \eqref{eq pi ek}, the character of the representation $\pi$ on the space \eqref{eq ker C plus} equals the series \eqref{eq sum char}, which converges to $\Theta$. 
\end{proof}
}

\begin{remark}
\label{rem:L^2.kernel}
The $L^2(\C, \psi)$-kernel of $\bar\partial+\bar\partial^*$ can be identified as the $L^2$-kernel of a deformed operator. For example, let $\psi(z) = e^{-|z|^2/2}$. Recall that $\bar\partial+\bar\partial^*$ is an operator on $\Omega^{0,*}(\C)$, given by
\[
\bar\partial+\bar\partial^*=c(dz)\frac{\partial}{\partial z}+c(d\bar z)\frac{\partial}{\partial\bar z},
\]
where now $c(d\bar z)=\frac{1}{\sqrt{2}}\ext(d\bar z)$ and $c(dz)=-\frac{1}{\sqrt{2}}\Int(dz).$ (See~\cite[Section~3.6]{BGV}.)
Set
\[
b:= \frac{1}{2} z c(d\bar z).
\]
Then $b^* = -\frac{1}{2} \bar z c(dz)$.
%
We have the deformed operator
\[
\begin{split}
\bar\partial +b=c(d\bar z)\left(\frac{\partial}{\partial\bar z}+\frac{z}{2} \right)\colon& \Omega^{0,0}(\C) \to \Omega^{0,1}(\C);\\
(\bar\partial +b)^*=c(dz)\left(\frac{\partial}{\partial z}-\frac{\bar z}{2} \right)\colon& \Omega^{0,1}(\C) \to \Omega^{0,0}(\C).
\end{split}
\]
The operator $U\colon \Omega^{0, *}(\C)\rightarrow\Omega^{0, *}(\C, \psi)$ given by $U(\alpha)=\psi^{-1}\alpha$ is a unitary isomorphism.
We have
\[
\bar\partial U(f)=\bar\partial(\psi^{-1}f)=\psi^{-1}\bigl(\bar\partial+\frac{z}{2}\bigr)f=U\left(\bigl(\bar\partial+\frac{z}{2}\bigr)f\right).
\]
Similarly, $U$ intertwines $\bar\partial^*$ and $(\bar\partial+b)^*$.
It then follows that
\[
\begin{split}
\ker_{L^2}(\bar\partial+b)&\cong\ker_{L^2, \psi}(\bar\partial); \\ 
\ker_{L^2}(\bar\partial+b)^*&\cong\ker_{L^2, \psi}(\bar\partial^*)=0.
\end{split}
\]
\end{remark}


\subsection{The circle acting on the two-sphere} \label{sec T S2}

As in Subsection \ref{sec:S^1.on.R^2}, we consider the circle group $\bT^1$, this time acting by rotations on the two-sphere $S^2$. In this compact setting, the usual index theory, and the Atiyah--Segal--Singer theorem apply. But we can use the $g$-index to decompose indices in this case.

We embed $\bT^1 \cong \SO(2)$ into $\SO(3)$ in the top-left corner. Then $S^2 = \SO(3)/\bT^1$. Identifying this space with $\bP^1(\C)$, we obtain a complex structure on it. Fix $n\in \Z_{\geq 0}$. Let $\C_n$ be the space of complex numbers, on which $\bT^1$ acts by
\[
g\cdot z = g^n z,
\]
for $g\in \bT^1$ and $z\in \C_n$. We have the line bundle
\[
L_n := \SO(3) \times_{\bT^1} \C_n \to S^2.
\]
Let $\bar \partial_n + \bar \partial_n^*$ be the Dolbeault--Dirac operator on $S^2$, coupled to $L_n$. Since $S^2$ is compact, we have the equivariant index
\[
\ind_{\SO(3)}(\bar \partial_n + \bar \partial_n^*) \in R(\SO(3)).
\]
By the Borel--Weil--Bott theorem, this index is the irreducible representation $V_n$ of $\SO(3)$ with highest weight $n$ (with respect to the positive root corresponding to the identification of $S^2$ with $\bP^1(\C)$).

Fix $g\in \bT^1$ with dense powers. By the Atiyah--Segal--Singer theorem, or Corollary \ref{cor fin fixed pt hol}, the character of $V_n$ evaluated at $g$ equals
\beq{eq ind T1 S2}
\ind_{\bT^1}(\bar \partial_n + \bar \partial_n^*)(g) = \frac{g^n}{1-g^{-1}} + \frac{g^{-n}}{1-g}.
\eeq
The two terms on the right hand side correspond to the two fixed points of the action by $\bT^1$. This expression can be rewritten as the finite sum
\[
\sum_{j=0}^{2n}g^{j-n}.
\]
This is the usual decomposition of $V_n|_{\bT^1}$ into irreducible representations of $\bT^1$.

So far, we have done nothing new in this example. But let $\bar \partial^{\C} + (\bar \partial^{\C})^*$ be the Dolbeault--Dirac operator on $\C$.
Then the linearisation theorem, Corollary \ref{cor lin}, implies that
\[
\ind_{\bT^1}(\bar \partial_n + \bar \partial_n^*)(g) = 
\ind_g \bigl( \bar \partial^{\C} + (\bar \partial^{\C})^* \bigr) g^n
+ \ind_{g^{-1}} \bigl( \bar \partial^{\C} + (\bar \partial^{\C})^* \bigr) g^{-n}.
\]
As we saw in Subsection \ref{sec:S^1.on.R^2}, Corollary \ref{cor fin fixed pt hol} implies that
\[
\ind_g \bigl( \bar \partial^{\C} + (\bar \partial^{\C})^* \bigr) = \frac{1}{1-g^{-1}}, 
\]
and similarly with $g$ replaced by $g^{-1}$. This agrees with \eqref{eq ind T1 S2}. Using Proposition~\ref{prop C psi}, we can realise the latter index as the character of the representation of $\bT^1$ in 
\[
\ker_{L^2, \psi} (\bar \partial + \bar\partial^*)^{+},
\]
with $\psi$ as in Subsection \ref{sec:S^1.on.R^2}.

\subsection{Discrete series characters} \label{sec ds}

Let $G$ be a connected, semisimple Lie group. Let $T<G$ be a maximal torus, and suppose it is a Cartan subgroup of $G$, i.e.\ $G$ has discrete series representations. Let $K<G$ be a maximal compact subgroup containing $T$. 
We denote the normalisers of $T$ in $G$ and $K$ by $N_G(T)$ and $N_K(T)$, respectively.
\begin{lemma} \label{lem fixed GT}
The fixed point set of the action by $T$ on $G/T$ is $N_K(T)/T$, the Weyl group  $W_c$ of $(\kk_{\C}, \kt_{\C})$.
\end{lemma}
\begin{proof}
Since
\[
(G/T)^T = N_G(T)/T,
\]
it is enough to show that
\[
N_G(T) = N_K(T).
\]
To prove this, 
let $\kg = \kp \oplus \kk$ be the Cartan decomposition of $\kg$.
Suppose $X\in \kp$, such that $\exp(tX) \in N_G(T)$ for all $t \in \R$. Then for all $H\in \kt$,
\[
\exp(tX)\exp(H)\exp(-tX) = \exp(\Ad(\exp(tX))H) \in T.
\] 
So
$
[X,H]\in \kt.
$
Because $X\in \kp$ and $H\in \kt \subset \kk$, we have $[X,H] \in \kp$. Hence $[X,H]=0$. Since $\kt$ is maximal commutative, we find that $X\in \kt$, so that $X=0$.
Therefore, an element $Y \in \kg$ such that $\exp(tY) \in N_G(T)$ for all $t \in \R$ must lie in $\kk$. Since $G$ is connected, the claim follows.
\end{proof}
\begin{example}
If $G = SL(2,\R)$, then a strongly elliptic coadjoint orbit of $G$ is equivariantly diffeomorphic to $G/T$. This is now a hyperbolic plane, on which $T$ acts by rotations. This action has one fixed point, corresponding to the trivial Weyl group of $K=T$.
\end{example}

Let $\lambda \in i\kt^*$ be regular (in the sense that $(\alpha, \lambda) \not=0$ for all roots $\alpha$, for a Weyl group invariant inner product). Fix a set $R^+$ of positive roots
 for $(\kg_{\C}, \kt_{\C})$ by defining a root $\alpha$ to be positive if $(\alpha, \lambda)>0$. Let $\rho$ be half the sum of the positive roots. 
 The choice of positive roots determines a $G$-invariant complex structure on the manifold $G/T$, defined by
\beq{eq compl str}
T_{eT}^{0,1}(G/T) = (\kg/\kt)^{0,1} := \bigoplus_{\alpha \in R^+} (\kg_{\C})_{-\alpha}.
\eeq
 
  Suppose $\lambda + \rho$ is an integral weight. Then $\lambda - \rho$ is integral as well, and we have
 the holomorphic line bundle
\[
L_{\lambda - \rho} := G\times_T \C_{\lambda - \rho} \to G/T,
\]
where $T$ acts on $\C_{\lambda - \rho} := \C$ via the weight $e^{\lambda - \rho}$. Let 
\[
\bar\partial_{L_{\lambda- \rho}} + \bar\partial_{L_{\lambda - \rho}}^* 
\]
be the Dolbeault--Dirac operator on $G/T$, coupled to $L_{\lambda- \rho}$. Let $g\in T$ be such that the powers of $g$ are dense in $T$. (Then in particular, $g$ is a regular element.) 

Let $\Theta_{\lambda}$ be the distributional character of the discrete series representation of $G$ with infinitesimal character $\lambda$.
\begin{proposition} \label{prop fixed GT}
One has
\[
\ind_{g}(\bar\partial_{L_{\lambda - \rho}} + \bar\partial_{L_{\lambda - \rho}}^*) = (-1)^{\frac{\dim G/K}{2}} \Theta_{\lambda}(g).
\]
\end{proposition}
\begin{proof}
The proof is analogous to Atiyah and Bott's derivation of the Weyl character formula from their fixed point theorem in {\cite[Section~5]{AB2}}. By Corollary~\ref{cor fin fixed pt hol} and Lemma~\ref{lem fixed GT}, we have
\beq{eq fixed GT}
\ind_{g}(\bar\partial_{L_{\lambda- \rho}} + \bar\partial_{L_{\lambda- \rho}}^*) = \sum_{aT \in N_K(T)/T} \frac{e^{\lambda- \rho}(a^{-1}ga)}
{\det (1-\Ad^{0,1}_{\kg/\kt}(a^{-1}g a))}.
\eeq
Here $\Ad^{0,1}_{\kg/\kt}\colon T\to \GL\bigl((\kg/\kt)^{0,1}\bigr)$ is induced by the adjoint representation. 
Because of \eqref{eq compl str},
we have
\[
\det (1-\Ad^{0,1}_{\kg/\kt}(a^{-1}g a))=  \prod_{\alpha \in R^+} (1-e^{-\alpha}(a^{-1}g a)).
\]

Since in the identification $N_K(T)/T = W_c$, the normaliser $N_K(T)$ acts on $i\kt^*$ via the coadjoint action, we find that \eqref{eq fixed GT} equals
\beq{eq sum W}
\sum_{w\in W_c}\frac{e^{w\cdot (\lambda- \rho)}}{\prod_{\alpha \in R^+} (1-e^{-w\cdot \alpha})}(g).
\eeq
Consider the Weyl denominator
\[
\Delta := e^{\rho}\prod_{\alpha \in R^+} (1-e^{- \alpha}).
\]
One has for all $w\in W_c$,
\[
w\cdot \Delta := e^{w\cdot \rho}\prod_{\alpha \in R^+} (1-e^{- w\cdot \alpha}) = \varepsilon(w) \Delta,
\]
where $\varepsilon(w) = \det w$ is the sign of $w$. Hence
we find that \eqref{eq sum W} equals
\[
\frac{ \sum_{w\in W_c}\varepsilon(w) e^{w\cdot \lambda }}{\Delta}(g).
\]
(This expression still makes sense  if $\rho$ is not an integral weight.)
By Harish-Chandra's character formula for the discrete series (see \cite[Theorem 16]{HC} or \cite[Theorem~12.7]{Knapp}), this equals $(-1)^{\frac{\dim G/K}{2}} \Theta_{\lambda}(g)$.
\end{proof}

Note that Proposition \ref{prop fixed GT} only relates the value of the character $\Theta_{\lambda}$ at $g$ to the $g$-index of 
$\bar\partial_{L_{\lambda - \rho}} + \bar\partial_{L_{\lambda - \rho}}^*$ if $g$ is a regular element of a maximal torus. Such elements form an open subset of $G$, and characters are not determined by their restrictions to this set. However, we can still use Proposition \ref{prop fixed GT}
 to give a description of the $g$-index in terms of the kernel of $(\bar\partial_{L_{\lambda- \rho}} + \bar\partial_{L_{\lambda- \rho}}^*)$.
\begin{proposition}
Suppose $G$ is a linear group. Then
the representation of $G$ in the $L^2$-kernel of $(\bar\partial_{L_{\lambda- \rho}} + \bar\partial_{L_{\lambda- \rho}}^*)^{\pm}$ has a distributional character $\Theta^{\pm}$ that can be evaluated at $g$, and one has
\[
\ind_{g}(\bar\partial_{L_{\lambda- \rho}} + \bar\partial_{L_{\lambda- \rho}}^*) = \Theta^+(g) - \Theta^-(g).
\]
\end{proposition}
\begin{proof}
This follows from Proposition \ref{prop fixed GT} and Schmid's realisation of the discrete series in the $L^2$-Dolbeault cohomology of $G/T$ with values in $L_{\lambda- \rho}$, in~\cite[Theorem~1.5]{Schmid}. Schmid's result implies that space
\[
\ker_{L^2}(\bar\partial_{L_{\lambda- \rho}} + \bar\partial_{L_{\lambda- \rho}}^*)^{\pm} 
\] 
equals zero if $\pm = -(-1)^{\frac{\dim G/K}{2}}$, and the representation of $G$ in this space is the discrete series representation with infinitesimal character $\lambda$ if  $\pm = (-1)^{\frac{\dim G/K}{2}}$. (The integer $k$ in Schmid's result now equals $\dim(G/K)/2$, and his $\lambda$ is our $\lambda - \rho$.) Hence
\[
\Theta^+- \Theta^- = (-1)^{\frac{\dim G/K}{2}} \Theta_{\lambda}.
\]
So the claim follows from Proposition \ref{prop fixed GT}.
\end{proof}

\subsection{Fredholm operators} \label{sec Fredholm}

For Fredholm operators, 
it is a natural question how the $g$-index of such an operator is related to the traces of $g$ acting on even and odd parts of its kernel. This depends on the behaviour of the operator `towards infinity'. To make this more explicit, let $M^+$ be the one-point compactification of $M$. The point at infinity is fixed by $g$. Let $U, V\subset M$ be as in Subsection \ref{sec g index}. Let $U', V' \subset M^+$ be $g$-invariant neighbourhoods of the point at infinity, such that $\overline{U'}\subset V'$, and $V\cap V' = \emptyset$. Then $U\sqcup U'$ and $V\sqcup V'$ are neighbourhoods of $(M^+)^g$ as in~\eqref{eq ind g}.
Lemma \ref{lem U V} therefore implies that for any $\sigma$-unital $G$-$C^*$ algebra $A$, the following diagram commutes:
\beq{eq diag Mplus}
\xymatrix{
KK_G(C(M^+), A)_g \ar[dd]_-{(p^{M^+}_*)_g}\ar[rrr]^-{(k^{M^+}_V)^*_g \oplus (k^{M^+}_{V'})^*_g} &&& KK_G(C_0(V), A)_g  \oplus KK_G(C_0(V'), A)_g \ar[d]_-{((j^V_{\Ubar})_*)_g^{-1} \oplus ((j^{V'}_{\overline{U'}})_*)_g^{-1}} \\
 &&& KK_G(C(\Ubar), A)_g \oplus KK_G(C(\overline{U'}), A)_g \ar[d]_-{(p^{\Ubar}_*)_g \oplus (p^{\overline{U'}}_*)_g} \\
 KK_G(\C, A)_g &&&\ar[lll]_-{+} KK_G(\C, A)_g \oplus KK_G(\C, A)_g.
}
\eeq
Indeed, since $M^+$ is compact, one can apply Lemma \ref{lem U V} to the pairs of neighbourhoods $U \sqcup U' \subset V \sqcup V'$ and $M^+ \subset M^+$ of $(M^+)^g$.

Now suppose that $(D^2+1)^{-1}$ is a compact operator. 
Then $F:= \frac{D}{\sqrt{D^2+1}}$ is Fredholm, so $\ker_{L^2}(D)$ is finite-dimensional. Let the representation $\pi_{M^+}\colon C(M^+)\to \cB(L^2(E))$ be defined by
\beq{eq def pi M+}
\pi_{M^+}(f+z) = \pi_M(f) + z,
\eeq
for $f\in C_0(M)$ and $z\in \C$. Then the triple $(L^2(E), F, \pi_{M^+})$ is a Kasparov $(C(M^+), \C)$-module. Let 
\beq{eq D Mplus}
_{M^+}[D] \in KK_G(M^+, \pt)
\eeq
be its class. In this case, we will write $\ind^{\infty}_g(D)$ for a version of the $g$-index of $D$ that captures the behaviour of $D$ at infinity:
\beq{eq inf index}
\ind^{\infty}_g(D) := (\ev_g)\circ (p^{\overline{U'}}_*)_g \circ ((j^{V'}_{\overline{U'}})_*)_g^{-1}\circ (k^{M^+}_{V'})^*_g (_{M^+}[D]_g).
\eeq
\begin{proposition}\label{prop Fredholm}
If $(D^2+1)^{-1}$ is compact, then 
\beq{eq Fredholm}
\Tr(g \, \mathrm{on}\, \ker_{L^2}(D^+)) - \Tr(g \, \mathrm{on}\,  \ker_{L^2}(D^-)) = \ind_g(D) + \ind_g^{\infty}(D).
\eeq
\end{proposition}
\begin{proof}
By commutativity of \eqref{eq diag Mplus}, with $A = \C$, we have
\beq{eq Fredholm 1}
(\ev_g)_g \circ (p^{M^+}_*)_g [D] = \ind_g(D) + \ind_g^{\infty}(D).
\eeq
Now
\[
p^{M^+}_* [D] = [L^2(E), F] = [\ker F, 0]\quad \in KK_G(\pt, \pt),
\]
so the left hand sides of \eqref{eq Fredholm} and \eqref{eq Fredholm 1} are equal.
\end{proof}

In concrete situations, knowledge of $\ind_g^{\infty}(D)$ then allows one to use the fixed point formula in Theorem \ref{thm fixed pt} to compute the left hand side of~(\ref{eq Fredholm}).

This can be made more explicit in a situation relevant to
 the treatment of Callias-type deformations of Dirac operators in the context of $KK$-theory in \cite{Bunke, Kucerovsky}. Suppose that $\Phi \in \End(E)^G$ is an odd, self-adjoint vector bundle endomorphism. Suppose that $\Phi^2 - 1_{E}$ tends to zero at infinity, so that it is a compact operator on $\Gamma_0(E)$. Then $(\Gamma_0(E), \Phi, \pi_{M^+})$ is an equivariant Kasparov $(C(M^+), C_0(M))$-cycle. Let $[\Phi] \in KK_G(M^+, M)$ be its class. Now we do not assume that $(D^2+1)^{-1}$ itself is compact, but that 
\[
[D_{\Phi}] := [\Phi]\otimes_{M}[D] \quad \in KK_G(M^+, \pt)
\]
is the class of an elliptic operator $D_{\Phi}$ as in \eqref{eq D Mplus}. Then $(D_{\Phi}^2+1)^{-1}$ is compact. (The idea is that $D_{\Phi} = D + {\Phi}$ if $D{\Phi} + {\Phi}D$ is sufficiently small; Compare this with \cite[Proposition 2.18]{Bunke}.) By functoriality of the Kasparov product, we have for $U', V' \subset M$ as above,
\beq{eq inf index DA}
\ind_g^{\infty}(D_{\Phi}) = (\ev_g)_g \left( 
\left( (p^{\overline{U'}}_*)_g \circ ((j^{V'}_{\overline{U'}})_*)_g^{-1}\circ (k^{M^+}_{V'})^*_g [{\Phi}]_g
\right) \otimes_M[D]_g
\right).
\eeq
This expression has the advantage that ${\Phi}$ is a vector bundle endomorphism, which makes \eqref{eq inf index DA} easier to evaluate than \eqref{eq inf index}. In particular, if
 ${\Phi}^2 = 1_{E}$ on $V' \cap M$, then $(k^{M^+}_{V'})^* [{\Phi}] = 0$. In that case, Theorem \ref{thm fixed pt} and Proposition \ref{prop Fredholm} imply that
\beq{eq index DA}
\Tr(g \, \mathrm{on}\, \ker_{L^2}(D_{\Phi}^+)) - \Tr(g \, \mathrm{on}\,  \ker_{L^2}(D_{\Phi}^-)) = \int_{TM^g} \frac{\ch \bigl( [\sigma_{D_{\Phi}}|_{TM^g}](g)\bigr) \Todd(TM^g\otimes\C)}{\ch\bigl(\bigl[\Bigwedge N_{\C}\bigr](g)\bigr)}.
\eeq

\begin{example}
Let $M = \C^n$, and let $g$ be the diagonal action by $n$ nontrivial elements of $\U(1)$. Then $M^g = \{0\}$, and $N = \C^n$. Let $\beta_{\C^{n}} \in KK_G(\pt, \C^{2n})$ be the Bott element as in Definition \ref{def:fiberwise.Bott.ele}. Now the class $[D_1] \in KK_G(\C^{2n}, \pt)$ as in \eqref{eq:D_1.cycle} is the Dolbeault class of $\C^{2n}$. The Kasparov product
\[
\beta_{\C^n} \otimes_{\C^{2n}} [D_1] \quad \in KK_G(\pt, \pt)
\]  
is represented by the elliptic operator $D_B := B\otimes 1 + 1 \otimes D_1$ as in \eqref{eq def F0}. Hence $(D_B^2+1)^{-1}$ is a compact operator. In the proof of Lemma \ref{lem betaN D1}, we saw that the $L^2$-kernel of $D_B$ is spanned by the $g$-invariant function \eqref{eq ker DB}. So
\beq{eq ind DB}
\Tr(g \, \mathrm{on}\, \ker_{L^2}(D_B^+)) - \Tr(g \, \mathrm{on}\,  \ker_{L^2}(D_B^-)) = 1.
\eeq

On the other hand, let $b\in C^{\infty}(\R)$ be an odd function, with values in $[-1, 1]$, such that $b(x) = 1$ for all $x\geq 1$. If we replace $B(1+B^2)^{-1/2}$ by $b(B)$ in \eqref{eq def beta}, then the resulting class in $KK_G(\pt, \C^{2n})$ is the same class $\beta_{\C^n}$. But with this normalisation function, we have $b(B)^2 = 1$ outside the unit ball in $\C^n$. So 
\beq{eq infty ind DB}
\ind_g^{\infty}(D_B) = 0.
\eeq
Finally, by Corollary \ref{cor fin fixed pt hol}, with $F = \Bigwedge \widetilde{N}_{\C} = \Bigwedge \C^{2n}$, we have
\beq{eq g ind DB}
\ind_g(D_B) = 1.
\eeq
The equalities \eqref{eq ind DB}, \eqref{eq infty ind DB} and \eqref{eq g ind DB} illustrate Proposition \ref{prop Fredholm} in this case.
 \end{example}
 
 \begin{example}
 In the setting of Theorem \ref{thm index pair}, the index pairing $[F] \otimes_M [D] \in KK_G(\pt, \pt)$ is represented by a Fredholm operator $D_F$. Analogously to \eqref{eq inf index DA}, we have $\ind_g^{\infty}(D_F) = 0$, so that Proposition \ref{prop Fredholm} and Theorem \ref{thm fixed pt} yield an expression for $\Tr(g \, \mathrm{on}\, \ker_{L^2}(D_F^+)) - \Tr(g \, \mathrm{on}\,  \ker_{L^2}(D_F^-))$. But in this setting, the same expression follows directly from Theorem \ref{thm index pair}.
 \end{example}
 
See Remark \ref{rem Dfv} for the construction of a Fredholm operator $D_{fv}$ as a deformation of \emph{any} elliptic operator $D$, with $\ind_g^{\infty}(D_{fv})=0$.

\subsection{Braverman's index} \label{sec Braverman}

Suppose $X\in \kg$ such that $g = \exp X$. Let $X^M$ be the vector field on $M$ defined by $X$. Suppose $D$ is a Dirac-type operator, and consider the deformed operator
\[
D^f_X := D + \ii fc(X^M).
\]
Here $f\in C^{\infty}(M)^G$, and $c\colon TM \to \End(E)$ is a given Clifford action, used to define the Dirac operator $D$. Braverman obtained a fixed point theorem for such operators, in \cite[Theorem~7.5]{Braverman}. This implies that the $g$-index equals Braverman's index in this case. 
\begin{corollary} \label{cor Br}
If $f$ is admissible (Definition 2.6 in \cite{Braverman}), then the representation of $G$ in 
$
\ker_{L^2}(D_X^f)^{\pm}
$
has a character $\chi^{\pm}$ that can be evaluated at $g$, and one has
\[
\ind_g(D) = \chi^+(g) - \chi^-(g).
\] 
\end{corollary}
\begin{proof}
The fixed point formula for $\ind_g(D)$ in Theorem \ref{thm fixed pt} is precisely the evaluation at $g$ of the right hand side of the formula in \cite[Theorem 7.5]{Braverman}. (This equality also shows that $\ker_{L^2}(D^f_X)$
has a character  that can be evaluated at~$g$.)
\end{proof}

\begin{remark}
In the above construction, the element $X \in \kg$, which represents the \emph{taming map} used in \cite{Braverman}, depends on the group element $g$. So the $g$-index of $D$ is not the character of the Braverman index of $D$ deformed by a single taming map, but the taming map depends on $g$.
\end{remark}

\iftoggle{long}
{
\begin{example}
The example in Subsection~\ref{sec:S^1.on.R^2} can be put in this context.
Let $\{ e_1, e_2 \}$ be the standard basis of $\R^2$. Consider the Clifford action by $T\R^2$ on $\R^2 \times \C^2$ defined by $c(e_1) :=\begin{bmatrix} 0 & -1\\ 1& 0\end{bmatrix}$ and $c(e_2) :=\begin{bmatrix}0 & \ii\\ \ii & 0\end{bmatrix}$. Consider the Dirac operator 
\[
D = c(e_1)\frac{\partial}{\partial x}+c(e_2)\frac{\partial}{\partial y} = 
\frac12\begin{bmatrix}0 & -\frac{\partial}{\partial z}\\ \frac{\partial}{\partial\bar z} & 0\end{bmatrix}
\]
on $\R^2 = \C$.
If $g = e^{\ii X}$, for $X \in \R = \Lie(\bT^1)$, then the vector field $X^M$ is given by 
\[
X^M_{(x, y)}=X \cdot (-y, x)\frac{1}{\sqrt{x^2+y^2}},
\]
for $(x, y) \in \R^2$.
 Choose the $\bT^1$-invariant function $f$ defined by $f(x, y) = -\frac{1}{4X}\sqrt{x^2+y^2}$. This function satisfies Braverman's admissibility criterion.\footnote{To check this, one can note that  the functions $X$ and $\|X^M\|$ are constant, and  compute that $\nabla^{LC}X^M$ is bounded. In (2.3) in \cite{Braverman}, the endomorphism $\mu^{E}$ is now zero, so that the function $\nu$ in (2.4) in \cite{Braverman} is bounded. Since $\|df\|$ is bounded and $f(x, y)\to \infty$ as $(x, y)\to \infty$, the admissibility criterion (2.5) in \cite{Braverman} follows.}
 For this choice of $f$, we have  
\[
-4fc(X^M)=c(-ye_1+xe_2)=\begin{bmatrix}0 & y+\ii x \\-y+\ii x & 0\end{bmatrix}.
\]
Therefore the deformed operator $D_X^f$ equals
\[
D+\ii fc(X^M)=\frac12\begin{bmatrix}0 & -\frac{\partial}{\partial z}+\frac{\bar z}{2} \\ \frac{\partial}{\partial\bar z}+\frac{z}{2} & 0\end{bmatrix}.
\]
In view of Remark~\ref{rem:L^2.kernel}, we have $\ker_{L^2}D_X^{\pm}\cong\ker_{L^2, \psi}(\bar\partial+\bar\partial^*)^{\pm}$. So the equality in Corollary \ref{cor Br} also follows from Proposition \ref{prop C psi} in this example. 
\end{example}
}{}

\subsection{A relative index theorem} \label{sec rel index}

In Theorem 4.18 of \cite{GL}, Gromov and Lawson obtain a relative index formula for pairs of elliptic operators that coincide outside compact sets. (See Theorem 2.18 in \cite{BrShi} for a version for Callias-type operators.) There is an analogue of this result for the $g$-index. 

For $j = 0,1$, let $M_j$ be a manifold with the same structure and properties as $M$. Let $E_j\to M_j$ be a vector bundle like $E\to M$, and let $D_j$ be an operator on $E_j$ like $D$ on $E$. Suppose there are relatively compact neighbourhoods $U_j$ of $M_j^g$ outside of which $M_j$, $E_j$ and $D_j$ can be identified. As on page 38 of \cite{GL}, we compactify $M_j$ to a manifold $\tilde M_j$, by taking a neighbourhood $V_j$ of $U_j$, and attaching a compact manifold $X$ to it. Since $M_0\setminus V_0 = M_1 \setminus V_1$, the same manifold $X$ can be used for $j = 0,1$. Extend the vector bundles $E_j$ and the operators $D_j$ to vector bundles $\tilde E_j\to \tilde M_j$ and elliptic operators $\tilde D_j$ on $\tilde E_j$. Suppose the map $g$ extends to $\tilde M_j$ and $\tilde E_j$, commuting with $\tilde D_j$.
\begin{proposition}[Relative index theorem]
We have
\[
\ind_g D_1 - \ind_g D_0 = \ind_G(\tilde D_1)(g) - \ind_G(\tilde D_0)(g). 
\]
\end{proposition}
Since the manifolds $\tilde M_j$ are compact, the indices
 on the right hand side of this equality are given by the usual equivariant index.
\begin{proof}
By the Atiyah--Segal--Singer fixed point theorem (or Theorem~\ref{thm fixed pt}), we have for $j=0,1$,
\begin{multline*}
\ind_G(\tilde D_j)(g)=
\int_{TM_j^g}
 \frac{\ch \bigl( [\sigma_{\tilde D_j}|_{TM_j^g}](g)\bigr) \Todd(TM^g_j\otimes\C)}{\ch\bigl(\bigl[\Bigwedge (N_j)_{\C}\bigr](g)\bigr)}\\
+
\int_{TX^g}
 \frac{\ch \bigl( [\sigma_{\tilde D_j}|_{TX^g}](g)\bigr) \Todd(TX^g\otimes\C)}{\ch\bigl(\bigl[\Bigwedge (N_X)_{\C}\bigr](g)\bigr)}. 
\end{multline*}
Here $N_j\to M_j^g$ and $N_X \to X^g$ are normal bundles to fixed point sets.
Since $\sigma_{\tilde D_1}|_{TX^g} =  \sigma_{\tilde D_0}|_{TX^g}$, Theorem \ref{thm fixed pt} implies the claim.
\end{proof}

\subsection{Some geometric consequences} \label{sec Hodge Spin}

The vanishing or non-vanishing of the $g$-index has some geometric consequences in the cases of Hodge-de Rham and $\Spin$-Dirac operators.

Let $D=d+d^*\colon \Omega^{\even}_{\C}(M) \to \Omega^{\odd}_{\C}(M)$ be the Hodge--de Rham operator on $M$, acting on complex differential forms.
The symbol class of this operator is $\bigl[\tau_M^*\Bigwedge TM_{\C}\bigr]$, whose restriction to $TM^g$ equals
\beq{eq Hodge}
[\sigma_D|_{TM^g}] = \bigl[\tau_{M^g}^*\Bigwedge{N_{\C}}\bigr] \otimes \bigl[\tau_{M^g}^*\Bigwedge{TM^g_{\C}}\bigr].
\eeq
Let $D_{M^g}$ be the component-wise defined Hodge--de Rham operator on $M^g$. Then
Theorem~\ref{thm fixed pt} and \eqref{eq Hodge} imply that
\beq{eq de Rham}
\ind_g(d+d^*)=\int_{TM^g}\ch(\sigma_{D_{M^g}})\Todd(TM^g\otimes\C)\\=\ind(D_{M^g})=\chi(M^g),
\eeq
the Euler characteristic of $M^g$. (See also \cite[p.\ 262]{LM}.)
\begin{corollary}
If $\ind_g(d+d^*) \not= 0$, then every $g$-invariant vector field on $M$ has a zero on $M^g$.
\end{corollary}
\begin{proof}
A $g$-invariant vector field on $M$ restricts to a vector field on $M^g$. If it does not vanish there, then $\chi(M^g)=0$. So the claim follows from \eqref{eq de Rham}.
\end{proof}

%

Next, suppose that $M$ is a $\Spin$ manifold, and that $g$ lifts to the spinor bundle. Let $D$ be the $\Spin$-Dirac operator.
\begin{corollary}
If $G$ is connected, $M$ is noncompact, and $\ind_g(D) = 0$, then the one-point compactification $M^+$ of $M$ is not a $G$-equivariant $\Spin$ manifold
\end{corollary}
\begin{proof}
If $M^+$ is a $G$-equivariant $\Spin$ manifold, with Dirac operator $D_{M^+}$, then the vanishing result of Atiyah and Hirzebruch \cite{AH} implies that
\[
 0 = \ind_g(D_{M^+}) = \ind_g(D) + a_{\infty}.
\]
Here $a_{\infty}$ is the contribution of the fixed point at infinity, which is nonzero by \cite[Theorem 8.35]{AB2}.
\end{proof}

\section{A non-localised index formula} \label{sec non-loc}

In the compact case, the Kirillov formula is a non-localised expression for the equivariant index of an elliptic operator; see {\cite[Theorem~8.2]{BGV}}. This can be deduced from the fixed point formula in the compact case. In the case of noncompact manifolds, there is also a non-localised expression for the $g$-index, Proposition~\ref{prop: Nonlocal formula} below. This follows from Kasparov's index theorem and the properties of the $g$-symbol class introduced in Subsection~\ref{sec g symbol}, rather than from Theorem~\ref{thm fixed pt}. 

A potentially interesting feature of this non-localised formula is that it involves the same kind of deformed symbols as the ones used for Dirac operators on symplectic manifolds in \cite{Paradan02}. Those deformed symbols are \emph{transversally elliptic} rather than elliptic. Berline and Vergne obtained a generalisation of the Atiyah--Segal--Singer fixed point formula to transversally elliptic operators or symbols, see \cite[Main Theorem 1]{BV1} and \cite[Theorem 20]{BV2}. This formula involves a distribution on the group. It was pointed out to the authors by Vergne that this formula implies that for the deformed symbols we will consider, at points $g$ where this distribution is given by a function, it is given by the $g$-index.

The index of such a deformed symbol was shown to equal the index of a deformed Dirac operator in Theorem 5.5 in \cite{Braverman}. In Theorem 1.5 in \cite{MZ}, this index is proved to be equal to another index of deformed Dirac operators, defined using the Atiyah--Patodi--Singer index on manifolds with boundary. 
In contrast to \cite{Braverman, MZ, Paradan02}, the expression for
 the $g$-index in terms of deformed symbols is independent of the choices made in this deformation. Furthermore, it applies to more general elliptic operators than Dirac operators.


We shall describe the $g$-symbol class ${\sigDg}$ of Definition \ref{def g symbol} more explicitly, using a deformed symbol. 
Let $v$ be a $G$-invariant vector field on $M$ that does not vanish outside $V$. 
\begin{example}
If $X\in \kg$ such that $g = \exp(X)$, one can take the vector field $v$ induced by $X$. This vector field obviously depends on $g$. 
\end{example}
\begin{example}
If $H$ is a compact Lie group acting on $M$, containing $G$, then it can be possible to choose a single vector field $v$ that works for all elements of $H$. Indeed, suppose there is an $H$-equivariant map $\psi\colon M\to \kh$, and consider  the Kirwan vector field $v$, defined by
\[
v_{m} := \left. \frac{d}{dt}\right|_{t=0} \exp(t\psi(m))\cdot m,
\]
for $m \in M$. Suppose this vector field is nonzero outside a compact set. Then $\psi$ is a taming map as in 
Definition 2.4 in \cite{Braverman}. In this case, the vector field $v$ can be used for any element of $H$.
\end{example}


Let $f\colon V\rightarrow\R_{\ge0}$ be a $G$-invariant continuous function, such that $f(m)=0$ when $m\in U$ and 
$\lim_{m\to m'}f(m) = \infty$ if $m'\in\partial V$.
Consider the deformed symbol $\sigma_{D, fv} \in \End(\tau_V^*(E|_V))$, given by
\beq{eq sigma D fv}
\sigma_{D, fv} (u) := \sigma_D(u+f(m)v_m)
\eeq
for $m\in V$ and $u\in T_mM$. Set 
\[
\tilde \sigma_{D, fv}  := \frac{\sigma_{D,fv}}{\sqrt{\sigma_{D,fv}^2+1}}.
\]
This defines an odd, self-adjoint, bounded operator on the Hilbert $C_0(TV)$-module $\Gamma_0(\tau_V^*(E|_V))$. Furthermore, we have for every vector field $u$ on $M$, and every $m'\in \partial V$,
\[
\lim_{m\to m'} \tilde \sigma_{D, fv}  (u_m) = \sgn(\sigma_D(v_{m'})).
\]
We extend $\tilde \sigma_{D, fv}  $ to a continuous vector bundle endomorphism of $\tau_M^*E$ by setting
\[
\tilde \sigma_{D, fv}  (u) := \sgn(\sigma_D(v_{m}))
\]
for all $u\in T_mM$, where $m\in M\setminus V$. (Since $v_m \not=0$ if $m\in M\setminus V$, this operator is invertible outside $V$.)

Note that 
\beq{eq:KK.w}
\tilde \sigma_{D, fv}  (u)^2-1\to 0
\eeq
as $u\to \infty$ in $TM$. Indeed, let $m\in M$ and $u\in T_mM$ be given. If $m\not\in V$, then $v_m\not=0$ and $\tilde \sigma_{D, fv}  (u)^2 = 1$. And if $m\in V$, then
\[
\tilde \sigma_{D, fv}  (u)^2-1 = \bigl(\sigma_D(u+f(m)v_m)^2+1\bigr)^{-1}.
\]
Since $D$ is elliptic and has first order, this
goes to zero as $u\to \infty$ in $TV$. We therefore find that $(\Gamma_0(\tau_M^*E), \tilde \sigma_{D, fv}  )$ is a Kasparov $(\C, C_0(TM))$-cycle. Let 
\[
{}_{\pt}[\sigma_{D, fv} ] \in KK_G(\pt, TM)
\]
be its class, which will be called the deformed symbol class.



\begin{lemma}
\label{lem:form of sigma}
The localisation of the deformed symbol class at $g$ is $\sigma_g^D$, i.e., 
\[
_{\pt}[\sigma_{D, fv} ]_g=\sigma_g^D \quad \in KK_G(\pt, TM)_g.
\]
\end{lemma}
\begin{proof}
As in Subsection \ref{sec Fredholm}, let $M^+$ be the one-point compactification of $M$. Let $U, V, U', V' \subset M^+$ be as in that subsection. Consider the class
\[
{}_{M^+}[\sigma_{D, fv}]:=[\Gamma_0(\tau_M^*E), \tilde\sigma_{D, fv}, \pi_{M^+}] \in KK_G(M^+, TM),
\]
where $\pi_{M^+}$ is as in \eqref{eq def pi M+}. Then by commutativity of \eqref{eq diag Mplus}, for $A = C_0(TM)$, we have
\begin{align}
_{\pt}[\sigma_{D, fv} ]_g &= (p^{M^+}_*)_g ({}_{M^+}[\sigma_{D, fv}]) \nonumber\\
	&= (p^{\Ubar}_*)_g \circ ((j^{V}_{\Ubar})_*)_g^{-1} \circ (k^{M^+}_V)^*_g ({}_{M^+}[\sigma_{D, fv}]_g) \label{eq sigma D g 1}\\
	&\qquad + (p^{\overline{U'}}_*)_g \circ ((j^{V'}_{\overline{U'}})_*)_g^{-1} \circ (k^{M^+}_{V'})^*_g ({}_{M^+}[\sigma_{D, fv}]_g). \nonumber
\end{align}
Now since $f = 0$ on $V$, we have
\[
 (k^{M^+}_V)^* ({}_{M^+}[\sigma_{D, fv}]) = (k^M_V)^*[\sigma_D].
\]
So the first term in \eqref{eq sigma D g 1} equals $\sigma^D_g$. Furthermore, 
\[
(k^{M^+}_{V'})^* ({}_{M^+}[\sigma_{D, fv}]) = \bigl[\Gamma^{\infty}(E|_{V'}), \sgn(\sigma_D(v)), \pi_{V'} \bigr]  = 0,
\]
since this class is represented by a degenerate cycle.
\end{proof}

\begin{remark}
Instead of 
\eqref{eq sigma D fv}, we could have used a more general deformed symbol of the form
\[
\sigma_{D, f\Phi}(u) := \sigma_D(u) + f(m)\Phi_m,
\]
for $m\in M$, $u\in T_mM$ and a $G$-equivariant, fibrewise self-adjoint, odd vector bundle endomorphism $\Phi$ of $E$, which is invertible outside $V$. We have used the natural choice $\Phi = \sigma_D(v)$.
\end{remark}

The realisation of the $g$-symbol class in Lemma \ref{lem:form of sigma} leads to the following non-local formula for the $g$-index. 
\begin{proposition}[Non-localised formula for the $g$-index]
\label{prop: Nonlocal formula}
The $g$-index of $D$ is calculated by 
\begin{equation}
\label{eq: Nonlocal formula}
\ind_g(D) 
=\bigl( _{\pt}[\sigma_{D, fv} ] \otimes_{TM}[D_{TM}]\bigr)(g).
\end{equation}
\end{proposition}
\begin{proof}
It follows from Definitions \ref{def:g-index} and \ref{def g symbol}, and Theorem~\ref{thm index}, that
\[
\ind_g(D) =\bigl({\sigDg}\otimes_{TM} [D_{TM}]_g\bigr)(g)
\]
The claim therefore follows from Lemma~\ref{lem:form of sigma}.
\end{proof}

\begin{remark} \label{rem Dfv}
Recall that when $M$ is noncompact, $\ind_g(D)$ is defined using $KK$-functorial maps. In Proposition \ref{prop: Nonlocal formula}, the class
\[
_{\pt}[\sigma_{D, fv} ]\otimes_{TM}[D_{TM}] \in KK_G(\pt, \pt)
\]
is represented by a Fredholm operator $D_{fv}$, defined in terms of the deformed symbol $\sigma_{D, fv}$ and the Dolbeault--Dirac operator $D_{TM}$.
Proposition~\ref{prop: Nonlocal formula} states that
\beq{eq:nonlocal&local}
\ind_g(D) = \Tr(g \, \mathrm{on}\, \ker_{L^2}(D_{fv}^+)) - \Tr(g \, \mathrm{on}\,  \ker_{L^2}(D_{fv}^-)).
\eeq
Theorem~\ref{thm fixed pt} then yields a cohomological expression for the right hand side of~(\ref{eq:nonlocal&local}).
(Note the analogy with \eqref{eq index DA}; we now have $\ind_g^{\infty}(D_{fv}) = 0$.) 
\end{remark}


\begin{thebibliography}{99}

\bibitem{Anghel} N.\ Anghel,  {\it On the index of Callias-type operators},  Geom.\ Funct.\
Anal.\ 3 (5) (1993), 431--438.
\bibitem{AB2} M.F.\  Atiyah and R.\ Bott,
{\it A Lefschetz fixed point formula for elliptic complexes. II. Applications.} 
Ann.  Math. (2) 88 (1968), 451--491. 
\bibitem{AH} 
M.F.\ Atiyah, F.\ Hirzebruch,  {\it Spin-manifolds and group actions},  1970 Essays on Topology and Related Topics (M\'emoires d\'edi\'es ˆ Georges de Rham) Springer, New York, 18--28.
\bibitem{AS} M.F.\ Atiyah and G.B.\ Segal, {\it The index of elliptic operators: II}, Ann.\ Math.\  87 (3) (1968), 531--545.
\bibitem{AS1} M.F.\ Atiyah and I.M.\ Singer, {\it The index of elliptic operators: I}, Ann.\ Math.\ 87 (3) (1968), 484--530.
\bibitem{AS3} M.F.\ Atiyah and I.M.\ Singer, {\it The index of elliptic operators: III}, Ann.\ Math.\  87 (3) (1968), 546--604.
\bibitem{BCH} {P. Baum, A. Connes and N. Higson}, {\it Classifying space for proper actions and $K$-theory of group $C^*$-algebras}, { Contemp. math. }167 (1994), 241--291.
\bibitem{BGV} {N. Berline, E. Getzler and M. Vergne}, {\it Heat kernels and Dirac operators}, Grundlehren der math.\  Wissenschaften vol.\ 298, Springer (1992).
\bibitem{BV1} N. Berline and M. Vergne,  {\it
The Chern character of a transversally elliptic symbol and the equivariant index}, 
Invent. Math. 124,  (1--3) (1996), 11--49. 
\bibitem{BV2} N. Berline and M.  Vergne, {\it
L'indice Žquivariant des opŽrateurs transversalement elliptiques}, 
Invent. Math. 124,  (1--3) (1996), 51--101. 
\bibitem{Blackadar} B.\ Blackadar, {\it $K$-theory for operator algebras}, second edition, Mathematical Sciences Research Institute Publications, vol.\ 5, Cambridge University Press (1998).
\bibitem{Braverman} M.\ Braverman, {\it Index theorem for equivariant Dirac operators on non-compact manifolds}, K-Theory, 27 (1) (2002), 61--101.
\bibitem{Braverman14} M.\ Braverman, {\it The index theory on non-compact manifolds with proper group action}, J.\ Geom.\ Phys., 98 (2015), 275--284.
\bibitem{BrShi} M.\ Braverman and P.\ Shi, {\it Cobordism invariance of the index of Callias-type operators}, ArXiv:1512.03939.
\bibitem{Bunke} U.\ Bunke, {\it A $K$-theoretic relative index theorem and Callias-type Dirac operators}, Math.\ Ann. 303 (2) (1995), 241--279
\bibitem{Callias} C.\ Callias, {\it Axial anomalies and index theorems on open spaces}, Comm.\ Math.\ Phys.\ 62 (1978), 213--234.
\bibitem{CGRS} A.L.\ Carey, V.\ Gayral, A.\ Rennie and F.A.\ Sukochev, {\it Index theory for locally compact noncommutative geometries}, Mem.\  Amer.\ Math.\ Soc.\ 231 (1085) (2014), vi+130 pp.
\bibitem{DEM} I.\ Dell'Ambrogio, H.\ Emerson and R.\ Meyer, {\it An equivariant Lefschetz fixed-point formula for correspondences}, Doc.\ Math.\ 19 (2014), 141--194. 
\bibitem{Emerson} H.\ Emerson, {\it Duality, correspondences and the Lefschetz map in equivariant $KK$-theory: a survey},  Perspectives on noncommutative geometry, 
Fields Inst. Commun., 61, Amer. Math. Soc., Providence, RI (2011), 4--78. 
\bibitem{GL} M.\ Gromov and H.B.\ Lawson, {\it Positive scalar curvature and the Dirac operator on complete Riemannian manifolds}, Pub.\ Math.\ I.H.\'E.S.\ 58 (1983), 83--196.
\bibitem{GGK} V.\ Guillemin, V.\ Ginzburg and Y.\ Karshon, {\it Moment maps, cobordisms and Hamiltonian group actions}, Mathematical surveys and monographs vol.\ 98, American Mathematical Society (2002).
\bibitem{HC} Harish--Chandra {\it Discrete series of semisimple Lie groups II, explicit determination of the characters}, Acta Math.\ 116 (1966), 1--111.
\bibitem{HR} N.\ Higson and J.\ Roe, {\it Analytic $K$-homology}, {Oxford mathematical monographs},
{Oxford University Press} {(2000)}.
 \bibitem{HM} P.\ Hochs and V.\ Mathai, {\it Geometric quantization and families of inner products},  Adv. Math. 282 (2015), 362--426.
\bibitem{HS} P.\ Hochs and Y.\ Song, {\it Equivariant indices of Spin$^c$-Dirac operators for proper moment maps}, ArXiv:1503.00801.
\bibitem{Kas13} G.G.\ Kasparov, {\it $K$-theoretic index theorems for elliptic and transversally elliptic operators}, J. Noncommut.\ Geom.\ (to appear), version November 2013.
\bibitem{Knapp} A.\ Knapp, {\it Representation theory of semisimple groups, an overview based on examples}, Princeton landmarks in math., Princeton Univ.\ Press (2001).
\bibitem{Kucerovsky} D.\ Kucerovsky, {\it A short proof of an index theorem}, Proc.\ Amer.\ Math.\ Soc.\ 129 (12) (2001), 3729--3736.
\bibitem{LM} H.B. Lawson Jr. and M.-L. Michelsohn, {\it Spin Geometry}, Princeton University Press (1989). 
\bibitem{MZ} X.\ Ma and W.\ Zhang, {\it Geometric quantization for proper
moment maps: the Vergne conjecture}, Acta Math., 212 (1) (2014), 11--57.
\bibitem{Palais} R.S.\ Palais, {\it On the existence of slices for actions of non-compact Lie groups}, Ann.\ Math., vol.\ 73, (2) (1961), 295--323.  
\bibitem{Paradan02} P.-\'E.\ Paradan. {\it  Formal geometric quantization II},  Pacific J.
Math., 253 (1) (2011), 169--211.
\bibitem{Schmid} W.\ Schmid, {\it $L^2$-cohomology and the discrete series}, Ann.\  Math.\  103 (3) (1976), 375--394.
\bibitem{WW} B.-L.\ Wang and H.\ Wang, {\it Localized Index and $L^2$-Lefschetz fixed point formula for orbifolds}, J.\ Differential Geom.\ (to appear), ArXiv:1307.2088.

\end{thebibliography}
\end{document}